\documentclass{article}
\pdfoutput=1
\usepackage[utf8]{inputenc}
\usepackage{fullpage}
\usepackage{authblk}
\usepackage{amsmath,amsfonts,amssymb,amsthm,enumerate,mdframed,hyperref,cleveref,algorithmic,algorithm,multirow,import,tikz,booktabs,xcolor,soul}
\usetikzlibrary{arrows,positioning, shapes} 
\usepackage{caption,subcaption}

\newtheorem{theorem}{Theorem}
\newtheorem{proposition}[theorem]{Proposition}
\newtheorem{lemma}[theorem]{Lemma}

\newtheorem{definition}[theorem]{Definition}
\newtheorem{corollary}[theorem]{Corollary}

\mathcode`*=\numexpr\mathcode`*-"2000\relax

\definecolor{darkgreen}{rgb}{0,0.6,0}

\providecommand{\F}{\mathbb{F}}
\providecommand{\N}{\mathbb{N}}
\providecommand{\Z}{\mathbb{Z}}

\DeclareMathOperator{\PG}{PG}

\newcommand{\cM}{\mathcal{M}}

\title{Lengths of divisible codes with restricted column multiplicities}
\author{Theresa K\"orner and Sascha Kurz}
\affil{University of Bayreuth, 95440 Bayreuth, Germany\\$\{$theresa.koerner,sascha.kurz$\}$@uni-bayreuth.de}

\date{}

\begin{document}

\maketitle

\begin{abstract}
  We determine the minimum possible column multiplicity of even, doubly-, and triply-even codes 
  given their length. This refines a classification result for the possible lengths of $q^r$-divisible 
  codes over $\mathbb{F}_q$. We also give a few computational results for field sizes $q>2$. Non-existence 
  results of divisible codes with restricted column multiplicities for a given length have applications e.g.\ in 
  Galois geometry and can be used for upper bounds on the maximum cardinality of subspace codes.

  \smallskip
  
\end{abstract}

\section{Introduction}
\label{sec_introduction}

Adding a parity check bit to a binary linear code yields an even code, i.e., all codewords have an even weight. Doubly-even 
binary linear codes, where the weights of the codewords are multiples of four, have been the subject of extensive research 
for decades, see e.g.\ \cite{doran2011codes}. Also linear codes where all occurring weights are divisible by eight, so-called
triply-even codes, have been studied in the literature, see e.g.\ \cite{betsumiya2012triply,miezaki2019support,rodrigues2022some}. 
The mentioned classes of linear codes are 
special cases of so-called $\Delta$-divisible codes where the weights of the codewords all are divisible by some 
integer $\Delta>1$. They have e.g.\ applications for the maximum size of partial $k$-spreads, i.e., sets of pairwise 
disjoint $k$-dimensional subspaces, see e.g.\ \cite{honold2018partial}. More concretely, the non-existence of $q^{k-1}$-divisible 
codes over $\mathbb{F}_q$ of a certain length implies an upper bound on the cardinality of partial $k$-spreads. In \cite{kiermaier2020lengths} 
the possible lengths of $q^r$-divisible codes over $\mathbb{F}_q$ have been completely characterized. However, on the constructive side 
some of these codes require a relatively large column multiplicity. In some applications upper bounds on the maximum possible column multiplicity 
are known. E.g.\ in the situation of partial $k$-spreads the codes have a maximum column multiplicity of one, i.e., the codes have to be projective. 
This special case has received quite some attention, see e.g.\ \cite{honold2019lengths,kurz2020no}. Here we ask more generally for the minimum 
possible column multiplicity of a $\Delta$-divisible code over $\mathbb{F}_q$ having length $n$. Those results imply classification results 
for the possible lengths of $\Delta$-divisible codes over $\mathbb{F}_q$ given any upper bound $\gamma$ on the allowed column multiplicity, refining
the results from \cite{kiermaier2020lengths}. A general parametric solution to this problem seems very unlikely, so that we solve the first few smallest 
cases in this paper. Our utilized arguments will mostly be of geometric nature so that we will use the geometric reformulation of linear codes as 
multisets of points in projective spaces. Non-existence results for divisible codes of a certain length have applications for covering and 
packing problems in Galois geometry, see e.g.\ \cite{etzion2014covering,etzion2020subspace}. They can also be used to improve upon the
so-called Johnson bound on the size of constant-dimension codes, see \cite{kiermaier2020lengths}, as well as more general mixed dimension subspaces codes, 
see \cite{honold2019johnson}.

The remaining part of this paper is structured as follows. In Section~\ref{sec_preliminaries} we introduce the necessary preliminaries and state our 
main result as Theorem~\ref{main_thm}, i.e., for the binary case we completely determine the minimum possible column multiplicity of $\Delta$-divisible 
codes of length $n$ for each $\Delta\in\{2,4,8\}$. Classification results for even and doubly-even binary codes are obtained in Section~\ref{sec_even_doubly_even} 
and used to conclude results for triply-even binary codes in Section~\ref{sec_8_div}. We draw a brief conclusion in Section~\ref{sec_conclusion}. 

All of the used arguments are completely theoretical and do not rely on 
any computer calculations. As a verification and continuation we present computational results in Section~\ref{sec_computational_results} in the appendix. As a small 
justification why some of our arguments are quite lengthy we also give some information on the combinatorial richness of a special case in Section~\ref{sec_data}.       

\section{Preliminaries}
\label{sec_preliminaries}

For a prime power $q$ let $\mathbb{F}_q$ be the finite field with $q$ elements. Let $V\simeq \mathbb{F}_q^v$  be 
a $v$-dimensional vector space over $\mathbb{F}_q$  and $\PG(v - 1, q)$ the projective space associated to it. By a 
$k$-space of $\PG(v - 1, q)$ we mean a $k$-dimensional linear subspace of $V$, also using the terms points, lines, planes, and
hyperplanes for $1$-, $2$-, $3$-spaces, and $(v - 1)$-spaces, respectively. We define a multiset $\cM$ of points 
via its point multiplicities $\cM(P)\in \N$ for each point $P$. We allow the addition, subtraction, and scaling with rational factors of 
multisets of points componentwise as long as the resulting point multiplicities are all natural integers. For an arbitrary 
subspace $K$ in $\PG(v-1,q)$ we define $\cM(K):=\sum_{P\le K} \cM(P)$, where we write $A\le B$ if $A$ is a subspace of $B$ and the summation 
is over all points $P$. With this, we define the cardinality or size of $\cM$ 
as $\#\cM:=\cM(V)$, i.e., as the sum over all point multiplicities $\cM(P)$. A multiset $\cM$ of points in $\PG(v-1,q)$ is called spanning 
if $\left\langle P\,:\, \cM(P)\ge 1\right\rangle_{\F_q}=\F_q^v$. The maximum occurring point multiplicity of $\cM$ is denoted by
$\gamma_1(\cM)$, or just $\gamma_1$ whenever $\cM$ is clear from the context. More generally, for each $1\le i\le v$ we denote by
$\gamma_i$, more precisely $\gamma_i(\cM)$, the maximum of $\cM(K)$ where $K$ runs over all $i$-spaces. E.g., $\gamma_v=\#\cM$. 

To each multiset $\cM$ of $n$ points in $\PG(v - 1, q)$ we can assign a $q$-ary linear code $C(\cM)$ defined by
a generator matrix whose $n$ columns consist of representatives of the $n$ points of $\cM$. It is well-known, see e.g.~\cite{dodunekov1998codes},  
that this relation between $C(\cM)$ and $\cM$ associating a full-length linear $[n,v]_q$ code with a multiset $\cM$ of $n$ points in $\PG(v-1,q)$ 
induces a one-to-one correspondence between classes of (semi-)linearly equivalent spanning multisets of points and classes of (semi-)linearly 
equivalent full-length linear codes. The maximum point multiplicity $\gamma_1$ of a multiset $\cM$ of points is the same as the maximum column 
multiplicity of the corresponding linear code $C$ (given an arbitrary generator matrix). So, $C$ is projective iff $\cM$ is indeed a set, i.e., 
$\cM(P)\le 1$ for all points $P$.

A linear code $C$ is said to be $\Delta$-divisible, where $\Delta\in\mathbb{N}_{\ge 1}$, if all of its weights are divisible by $\Delta$. A 
multiset $\cM$ of points is called $\Delta$-divisible iff its corresponding linear code $C(\cM)$ is $\Delta$-divisible. More directly, a multiset 
$\cM$ of points is $\Delta$-divisible iff we have $\cM(H)\equiv \#\cM \pmod \Delta$ for all hyperplanes $H$. As for binary linear codes we speak of 
even, doubly-even, and triply-even multisets of points in $\PG(v-1,2)$ when they are $2$-, $4$-, and $8$-divisible, respectively.

If $\cM$ is a multiset of points in $\PG(v-1,q)$ and $K$ some subset of all points (usually a subspace in $\PG(v-1,q)$), then $\cM|_K$ denotes the restriction 
of $\cM$ to $K$, i.e., $\cM|_K(P)=\cM(P)$ for all $P\le K$ and $\cM|_K(P)=0$ otherwise. If $K$ is a hyperplane then the restricted multiset $\cM|_K$ inherits 
divisibility with a smaller divisibility constant, see e.g.\ \cite[Lemma 7]{honold2018partial}.
\begin{lemma}
  \label{lemma_divisibility_hyperplane}
  Let $\cM$ be a $\Delta$-divisible multiset of points in $\PG(v-1,q)$. If $q$ divides $\Delta$, then $\cM|_H$ is $(\Delta/q)$-divisible for each hyperplane 
  $H$. 
\end{lemma}  
Of course we can apply the lemma recursively so that $\cM|_S$ is $\left(\Delta/q^i\right)$-divisible for each subspace $S$ of codimension $i$, i.e., dimension $v-i$, 
if $\Delta$ is divisible by $q^i$.

For an arbitrary subspace $K$ we denote by $\chi_K$ its characteristic function, i.e., $\chi_K(P)=1$ iff $P\le K$ and $\chi_K(P)=0$ otherwise. The support 
$\operatorname{supp}(\cM)$ of a multiset of points $\cM$ is the set of points that have non-zero multiplicity.

For a given multiset $\cM$ of points in $\PG(v-1,q)$ we denote by $a_i$ the number of hyperplanes $H$ such that $\cM(H)=i$. If $\cM$ is spanning, then 
we have $a_{\#\cM}=0$. We say that $\cM$ has dimension $k$ if its span is a $k$-dimensional subspace $K$. By considering $\cM$ 
restricted to $K\cong \PG(k-1,q)$ we can always assume that $\cM$ is spanning if we choose a suitable integer for $k$. For the ease of notation we 
assume that $\cM$ is spanning in $\PG(k-1,q)$ in the following. Counting the number of hyperplanes in $\PG(k-1,q)$ gives   
\begin{equation}
  \sum_{i} a_i=\frac{q^k-1}{q-1}\label{se1}
\end{equation} 
and counting the number of pairs of points and hyperplanes gives
\begin{equation}
  \sum_{i} ia_i=\#\cM\cdot \frac{q^{k-1}-1}{q-1}.\label{se2}
\end{equation} 
By $\lambda_i$ we denote the number of points $P$ such that $\cM(P)=i$, so that
\begin{equation}
  \sum_{i} \lambda_i=\frac{q^k-1}{q-1}\label{se3}
\end{equation}
and
\begin{equation}
  \sum_{i} i\lambda_i =\#\cM.\label{se4}
\end{equation}  
Double-counting the incidences between pairs of elements in $\cM$ and hyperplanes gives
\begin{equation}
  \sum_{i} {i \choose 2} a_i={{\#\cM}\choose 2}\cdot \frac{q^{k-2}-1}{q-1}+q^{k-2}\cdot \sum_{i} {i\choose 2}\lambda_i.\label{se5}
\end{equation}
We call the equations (\ref{se1})-(\ref{se5}) the \emph{standard equations} for multisets of points. If $\cM$ is a set of points, then Equation~(\ref{se5}) simplifies to
$$
  \sum_{i} {i \choose 2} a_i={{\#\cM}\choose 2}\cdot \frac{q^{k-2}-1}{q-1}
$$
and is complemented to the standard equations (for sets of points) by equations (\ref{se1}) and (\ref{se2}). We call the vector $\left(a_i\right)_{i\in\N}$ the \emph{spectrum} 
of $\cM$. As an abbreviation we set $[k]_q:=\left(q^k-1\right)/(q-1)$ for all $k\in\N$.

If all hyperplanes have the same multiplicity, then there is a well known classification of the corresponding multisets of points. In order to keep the paper self-contained we give 
a direct proof. 

\begin{lemma}
  \label{lemma_repeated_simplex}
  Let $\cM$ be a spanning multiset of points of cardinality $n$ in $\PG(k-1,q)$ such that every hyperplane $H$ has multiplicity $\cM(H)=s$. Then, 
  we have $\cM(P)=t$ for every point $P$, where $t=n/[k]_q$. If $k\ge 2$, then we additionally have $s=t[k-1]_q$.  
\end{lemma}
\begin{proof}
  If $k=1$, then we can choose $t=n$. The unique point $P$ then satisfies $\cM(P)=\#\cM=n=t$. Now assume $k\ge 2$. Equation~(\ref{se1}) gives $a_s=[k]_q$, so that    
  Equation~(\ref{se2}) yields $s[k]_q=n[k-1]_q$. If $k=2$, then we have $\cM(P)=s=t[k-1]_q=t$ for every point $P$ since every hyperplane is a point for $k=2$. For $k\ge 3$ 
  double-counting the points of $\cM$ via the hyperplanes that contain $P$ gives $\cM(P)=t$.
\end{proof}

\begin{proposition}
  \label{prop_simplex_repetition_special}
  Let $0\le l\le r$ be integers and $\cM$ be a $q^r$-divisible multiset of points in $\PG(v-1,q)$ of cardinality $n=q^l\cdot [r+1-l]_q$. Then, there exists 
  a $(r+1-l)$-space $K$ such that $\cM= q^l\cdot \chi_K$.
\end{proposition} 
\begin{proof}
  If $l=r$, then $\cM$ is $q^r$-divisible with cardinality $q^r$, so that $\cM$ corresponds to a $q^r$-fold point. If $l<r$, then we have $q^r<n<2q^r$ and all hyperplanes 
  that do not contain all points of $\cM$ have the same multiplicity so that we can apply Lemma~\ref{lemma_repeated_simplex}.
\end{proof}

\begin{corollary}
  \label{cor_simplex_repetition_special}
  Let $\cM$ be a $\Delta$-divisible multiset of points in $\PG(v-1,2)$ of cardinality $n$. 
  \begin{itemize}
    \item If $\Delta=2$ and $n=2$, then $\cM$ is the characteristic function of a double point.
    \item If $\Delta=2$ and $n=3$, then $\cM$ is the characteristic function of a line. 
    \item If $\Delta=4$ and $n=4$, then $\cM$ is the characteristic function of a $4$-fold point.
    \item If $\Delta=4$ and $n=6$, then $\cM$ is the characteristic function of a double line.
    \item If $\Delta=4$ and $n=7$, then $\cM$ is the characteristic function of a plane.
    \item If $\Delta=8$ and $n=8$, then $\cM$ is the characteristic function of an $8$-fold point.
    \item If $\Delta=8$ and $n=12$, then $\cM$ is the characteristic function of an $4$-fold line.
    \item If $\Delta=8$ and $n=14$, then $\cM$ is the characteristic function of a double plane.
    \item If $\Delta=8$ and $n=15$, then $\cM$ is the characteristic function of a solid.
  \end{itemize}   
\end{corollary}

If $\cM$ is a multiset of points and $Q$ a point in $\PG(v-1,q)$, where $v\ge 2$, then we can construct a multiset $\cM_Q$ in $\PG(v-2,q)$ by \emph{projection} 
trough $Q$, that is the multiset image under the map $P\mapsto \langle P,Q\rangle /Q$ setting $\cM_Q(L/Q)=\cM(L)-\cM(Q)$ for every line $L\ge P$ in $\PG(v-1,q)$. 
We directly verify the following properties:
\begin{lemma}   
  \label{lemma_projection}
  Let $\cM$ be a spanning $\Delta$-divisible multiset of points in $\PG(k-1,q)$, where $k\ge 2$, and let $\cM_Q$ arise from $\cM$ by projection through a point 
  $Q$. Then we have $\#\cM_Q=\#\cM-\cM(Q)$, $\cM_Q$ is $\Delta$-divisible, the span of $\cM_Q$ has dimension $k-1$, and $\gamma_1\!\left(\cM_Q\right)=\cM(L)-\cM(Q)$, 
  where $L$ is a line containing $Q$ and maximizing $\cM(L)$.
\end{lemma}  

In the binary case also $2^r$-divisible multisets of points of cardinality $2^{r+1}$ can be characterized easily for each $r\in\N$. We first state
an auxiliary result that is also used later on.
\begin{lemma}
  \label{lemma_lower_bound_space_mult}
  Let $\cM$ be a spanning multiset of points in $\PG(k-1,q)$ with cardinality $n$, $1\le l\le k-2$, and $K$ be an arbitrary $l$-dimensional subspace. If all hyperplanes 
  containing $K$ have cardinality $s$, then $\cM(K)=s-\frac{n-s}{q-1}+\frac{n-s}{q^{k-l-1}(q-1)}>s-\frac{n-s}{q-1}$. If all hyperplanes containing $K$ have cardinality at least 
  $s$, then $\cM(K)\ge s-\frac{n-s}{q-1}+\frac{n-s}{q^{k-l-1}(q-1)}>s-\frac{n-s}{q-1}$. 
\end{lemma}  
\begin{proof}
  Counting points via the hyperplanes containing $K$ yields
  $$
    [k-l]_q\cdot\left(s-\cM(K)\right)=[k-l-1]_q\cdot\left(n-\cM(K)\right)
  $$  
  in the first case. Solving for $\cM(K)$ yields
  $$
    \cM(K)=\frac{[k-l]_q\cdot s-[k-l-1]_q\cdot n}{q^{k-l-1}}=s-\frac{n-s}{q-1}+\frac{n-s}{q^{k-l-1}(q-1)}>s-\frac{n-s}{q-1}.
  $$
  In the second case the same reasoning yields
  $$
    \cM(K)\ge s-\frac{n-s}{q-1}+\frac{n-s}{q^{k-l-1}(q-1)}>s-\frac{n-s}{q-1}.  
  $$
\end{proof}

\begin{proposition}
  \label{prop_multiples_of_affine_spaces}
  Let $r\ge 1$ be an integer and $\cM$ be a $2^r$-divisible multiset of points in $\PG(v-1,2)$ with cardinality $2^{r+1}$. Then, either $\cM=2^{r+1}\cdot\chi_P$ for a point $P$ 
  or there exist subspaces $K$ and $E\le K$ with $r+1 \ge \dim(E)=\dim(K)-1\ge 1$ such that $\cM=2^{r+1-\dim(E)} \cdot \chi_{K\backslash E}=2^{r+1-\dim(E)} \cdot \chi_K-2^{r+1-\dim(E)} \cdot \chi_E$.
\end{proposition}
\begin{proof}
  Let $k$ be the dimension of the span of $\cM$. If $k=1$, then there clearly exists a point $P$ with $\cM=2^{r+1}\cdot\chi_P$. For $k\ge 2$ the 
  standard equations yield $a_0=1$ and $a_{2^r}=2^{k}-2$. We choose $E$ as the unique hyperplane with multiplicity zero (and $K$ as the entire ambient space). If $P$ is a point with $\cM(P)>0$, then 
  Lemma~\ref{lemma_lower_bound_space_mult} yields $\cM(P)=2^{r-k+2}$. Since there are exactly $2^{k-1}$ points outside of $E$, all points $P$ outside of $E$ 
  have multiplicity $2^{r-k+2}=2^{r+1-\dim(E)}$.
\end{proof}
In other words, the corresponding multisets of points are suitable multiples of affine spaces or are given by $2^{r+1}$-fold points, which might be considered as a degenerated case. 
We remark that the corresponding situation for $q>2$ is more complicated, see papers on the so-called cylinder conjecture \cite{de2019cylinder,kurz2021generalization}. 

\begin{definition} 
  By $\Gamma_q(\Delta,n)$ we denote the minimum of $\gamma_1(\cM)$ over all $\Delta$-divisible multisets 
  of points $\cM$ in $\PG(v-1,q)$ with cardinality $n$, where $v$ is sufficiently large. If no such multiset of points 
  exist we set $\Gamma_q(\Delta,n)=\infty$. 
\end{definition}

In \cite[Theorem 1]{ward1981divisible} it was shown that each $\left(p^ed\right)$-divisible code over a finite field with characteristic $p$, where $\gcd(p,d)=1$, is  
a $d$-fold repetition of a $p^e$-divisible code. So it 
suffices to determine $\Gamma_q(\Delta,n)$ for the cases when $\Delta$ has no non-trivial factor that is coprime to $q$. In \cite[Theorem 1]{kiermaier2020lengths} the possible (effective) 
lengths of $q^r$-divisible codes over $\F_q$ were completely characterized for all $r\in\N$. In order to state the result we need a bit more notation.  For each 
$r\in \N$ and each integer $0\le i\le r$ we define $s_q(r,i) := q^i\cdot [r-i+1]_q$. Note that the number $s_q(r,i)$ is divisible by $q^i$ but not by $q^{i+1}$. This
allows us to create kind of a positional system upon the sequence of base numbers $S_q(r) := \big(s_q(r,0),s_q(r,1),\dots,s_q(r,r)\big)$. With this, each integer $n$ 
has a unique \emph{$S_q(r)$-adic expansion}  
$$
  n = \sum_{i=0}^r e_i s_q(r,i)
$$
with $e_0,\dots,e_{r-1}\in \{0,\dots,q-1\}$  and \emph{leading coefficient} $e_r \in \Z$. Rewritten to our geometrical setting \cite[Theorem 1]{kiermaier2020lengths} says:
\begin{theorem}
  \label{thm_lengths_of_divisible_codes}
  For $n \in \Z$ and $r \in\N$ the following are equivalent:
  \begin{enumerate}
    \item[(i)] For sufficiently large $v$ there exists a $q^r$-divisible multiset of points of cardinality $n$ in $\PG(v-1,q)$.
    \item[(ii)] The leading coefficient $e_r$ of the $S_q(r)$-adic expansion of $n$ is non-negative.
  \end{enumerate}  
\end{theorem} 
As an example we consider the $S_2(2)$-adic expansion of $9$: $1\cdot 7 + 1\cdot 6 -1\cdot 4$. So, there is no $4$-divisible multiset of points with cardinality $9$ in $\PG(v-1,2)$, where the dimension
$v$ of the ambient space is arbitrary. A direct implication of Theorem~\ref{thm_lengths_of_divisible_codes} is:
\begin{proposition}
  \label{prop_Gamma_infty}
  We have $\Gamma_2(2,1)=\infty$, $\Gamma_2(4,n)=\infty$ for $n\in \{1,2,3,5,9\}$, and $\Gamma_2(8,n)=\infty$ for $n\in \{1, 2, 3, 4, 5, 6, 7, 9, 10, 11, 13, 17, 18, 19, 21, 25, 33\}$. 
\end{proposition}  
For $\Delta\in\{2,4,8\}$ the possible cardinalities of $\Delta$-divisible sets of points over $\F_2$ are completely determined, see e.g.\ \cite{honold2018partial} and 
\cite[Theorem 2]{honold2019lengths}, where it was shown that no binary $8$-divisible projective linear code of effective length $59$ exists. In our context this implies:
\begin{proposition}
  \label{prop_Gamma_1}
  We have $\Gamma_2(2,n)=1$ iff $n\ge 3$, $\Gamma_2(4,n)=1$ iff $n\in\{7,8\}$ or $n\ge 14$, and $\Gamma_2(8,n)=1$ iff $n\in \{15, 16, 30, 31, 32, 45, 46, 47, 48, 49, 50, 51\}$ 
  or $n\ge 60$. 
\end{proposition}  
   
We can use the corresponding examples to construct multisets of points of larger divisibility or larger cardinality. If $\cM$ is a $\Delta$-divisible multiset of points
in $\PG(v-1,q)$, then $q\cdot \cM$ is a $q\Delta$-divisible multiset of points in $\PG(v-1,q)$ with cardinality $q\cdot\#\cM$ and maximum point multiplicity $q\cdot\gamma_1(\cM)$. 
Using the decomposition $\F_q^{v_1}\oplus \F_q^{v_2}\cong \F_q^{v_1+v_2}$ we can also combine two $\Delta$-divisible multisets of points $\cM_i$ in $\PG(v_1-1,q)$, where 
$i=1,2$, to a $\Delta$-divisible multiset of points in $\PG(v_1+v_2-1,q)$ with cardinality $\#\cM_1+\#\cM_2$ and maximum point multiplicity $\max\!\left\{\gamma_1(\cM_1),\gamma_1(\cM_2)\right\}$. 
Applied recursively we obtain:
\begin{proposition}$\,$\\[-4.5mm]
  \begin{itemize}
    \item[(i)] We have $\Gamma_2(2,2)=2$, $\Gamma_2(4,n)=2$ for $n\in \{6,10,12,13\}$, and $\Gamma_2(8,n)=2$ for $n\in \{14, 28, 29, 34, 36, 38$, $40, 42, 43, 44, 52, 53, 54, 55, 56, 57, 58, 59\}$.
    \item[(ii)] We have $\Gamma_2(4,n)\le 4$ for $n\in\{4,11\}$ and $\Gamma_2(8,n)\le 4$ for $n\in \{12, 20, 24, 26, 27, 35, 39, 41\}$.   
    \item[(iii)] We have $\Gamma_2(8,n)\le 8$ for $n\in \{8, 22, 23, 37\}$.
   \end{itemize}
\end{proposition}  
The main goal of the remaining part of this paper is to show that the upper bounds in (ii) and (iii) are indeed sharp, see Theorem~\ref{main_thm}. We remark 
that the constructions used in \cite{kiermaier2020lengths} imply $\Gamma_q(q^r,n)\le r$ whenever $\Gamma_q(q^r,n)\neq \infty$ and $r\in\N$.  

\begin{theorem}
  \label{main_thm} $\,$\\[-4.5mm]
  \begin{itemize}
    \item We have $\Gamma_2(2,1)=\infty$, $\Gamma_2(4,n)=\infty$ for $n\in \{1,2,3,5,9\}$, and $\Gamma_2(8,n)=\infty$ for $n\in \{1, 2, 3, 4, 5, 6, 7, 9,$ $10, 11, 13, 17, 18, 19, 21, 25, 33\}$. 
    \item  We have $\Gamma_2(2,n)=1$ iff $n\ge 3$, $\Gamma_2(4,n)=1$ iff $n\in\{7,8\}$ or $n\ge 14$, and $\Gamma_2(8,n)=1$ iff $n\in \{15, 16, 30, 31, 32, 45, 46, 47, 48, 49, 50, 51\}$ 
           or $n\ge 60$.
    \item We have $\Gamma_2(2,2)=2$, $\Gamma_2(4,n)=2$ for $n\in \{6,10,12,13\}$, and $\Gamma_2(8,n)=2$ for $n\in \{14, 28, 29, 34, 36, 38$, $40, 42, 43, 44, 52, 53, 54, 55, 56, 57, 58, 59\}$.
    \item We have $\Gamma_2(4,n)= 4$ for $n\in\{4,11\}$ and $\Gamma_2(8,n)= 4$ for $n\in \{12, 20, 24, 26, 27, 35, 39, 41\}$.   
    \item We have $\Gamma_2(8,n)= 8$ for $n\in \{8, 22, 23, 37\}$.
   \end{itemize}
\end{theorem}

Another tool that we can use in the task of proving Theorem~\ref{main_thm} is the classification of $\Delta$-divisible codes spanned by codewords of weight
$\Delta$ \cite{kiermaier2020classification}. An exemplary implication is:
\begin{lemma}
  Let $\mathcal{C}$ be a binary code with non-zero weights in $\{8,16,24\}$ that is spanned by codewords of weight $8$. Then, we have
  $$
    A_8\in\{0,1,2,3,4,6,7,8,9,10,11,13,14,15,16,17,18,21,22,25,29,30,31,33,37,45\}
  $$
  for the number $A_8$ of the number of codewords of weight $8$ in $C$ (including the case that $C$ is empty). 
\end{lemma}
\begin{proof}
  We apply the classification of \cite{kiermaier2020classification}: 
  The possible weight enumerators of the indecomposable subcodes are given by $1+1x^8$, $1+3x^8$, $1+7x^8$, $1+15x^8$, $1+6x^8+1x^{16}$, $1+10x^8+5x^{16}$, 
  $1+14x^8+1x^{16}$, $1+30x^8+1x^{16}$, $1+15x^8+15x^{16}+1x^{24}$, and $1+21x^8+35x^{16}+7x^{24}$. Combining two such blocks with maximum weight $8$ gives the 
  further possibilities $1+2x^8+1x^{16}$, $1+4x^8+3x^{16}$, $1+8x^8+7x^{16}$, $1+2x^8+1x^{16}$, $1+16x^8+15x^{16}$, $1+2x^8+1x^{16}$, $1+6x^8+9x^{16}$, 
  $1+10x^8+21x^{16}$, $1+18x^8+45x^{16}$, $1+14x^8+49x^{16}$, $1+22x^8+105x^{16}$, and $1+30x^8+225x^{16}$. Combining these or a block with maximum weight $16$ 
  with a block with maximum weight $16$ further possibilities $1+7x^8+7x^{16}+1x^{24}$, $1+9x^8+19x^{16}+3x^{24}$, $1+13x^8+43x^{16}+7x^{24}$, $1+21x^8+91x^{16}+15x^{24}$, 
  $1+11x^8+15x^{16}+5x^{24}$, $1+7x^8+7x^{16}+1x^{24}$, $1+13x^8+35x^{16}+15x^{24}$, $1+17x^8+75x^{16}+35x^{24}$, $1+25x^8+155x^{16}+75x^{24}$, $1+15x^8+15x^{16}+1x^{24}$, 
  $1+17x^8+43x^{16}+3x^{24}$, $1+21x^8+99x^{16}+7x^{24}$, $1+29x^8+211x^{16}+15x^{24}$, $1+31x^8+31x^{16}+1x^{24}$, $1+33x^8+91x^{16}+3x^{24}$, $1+37x^8+211x^{16}+7x^{24}$, 
  and $1+45x^8+451x^{16}+15x^{24}$.         
\end{proof}

\section{Classification results for even and doubly-even multisets of points}
\label{sec_even_doubly_even}

A few classification results are already stated in Corollary~\ref{cor_simplex_repetition_special}. Here the characterized multisets of points are given by $\lambda\cdot\chi_K$ for some 
subspace $K$. For $n\ge 3$ another construction of a spanning even set of $n$ points in $\PG(n-2,2)$ is given by a so-called \emph{projective base} $B_n$ of size $n$, i.e., a set of $n$ points 
such that each $n-1$ points span an $(n-1)$-space. 

\begin{proposition}
  \label{prop_2_div_card_5}
  Let $\cM$ be a $2$-divisible multiset of points in $\PG(v-1,2)$ with cardinality $5$. Then either $\cM=\chi_L+2\cdot\chi_P$, where $L$ is a line and $P$ a point, or 
  $\cM$ is the characteristic function of a projective base $B_5$ of size $5$.
\end{proposition}
\begin{proof}
  If there exists a point $P$ with $\cM(P)\ge 2$, then $\cM-2\cdot\chi_P$ is also $2$-divisible with cardinality $3$, so that we can apply Proposition~\ref{prop_simplex_repetition_special}. 
  Thus, we can assume $\gamma_1(\cM)=1$ in the following so that the standard equations yield $k=4$ for the dimension of the span of $\cM$. Since $\gamma_1(\cM)=1$ no three points 
  form a line $L$ (since otherwise $\cM-\chi_L$ would be the characteristic function of a double point). Finally, $2$-divisibility implies that no four points can span a plane, which is a
  hyperplane in our situation.   
\end{proof}

\begin{proposition}
  \label{prop_2_div_card_4}
  Let $\cM$ be a $2$-divisible multiset of points in $\PG(v-1,2)$ with cardinality $4$. Then either $\cM=2\cdot \chi_{P_1}+2\cdot \chi_{P_2}$ for two points $P_1,P_2$ (that may also be 
  equal) or there exists a plane $E$ and line $L\le E$ with $\cM=\chi_E-\chi_L=\chi_{E\backslash L}$. 
\end{proposition}
\begin{proof}
  This is a special case of Proposition~\ref{prop_multiples_of_affine_spaces}.
\end{proof}

\begin{proposition}
  \label{prop_2_div_card_6}
  Let $\cM$ be a $2$-divisible multiset of points in $\PG(v-1,2)$ with cardinality $6$. Then either $\cM$ contains a point of multiplicity at least two, $\cM$ is the characteristic function 
  of a projective base $B_6$ of size $6$ or $\cM=\chi_{L_1}+\chi_{L_2}$ for two (disjoint) lines $L_1,L_2$.
\end{proposition}
\begin{proof}
  If there is no point of multiplicity at least two, then we have $\gamma_1(\cM)=1$, which we assume in the following. If $\cM$ contains three points forming a line $L$, then $\cM-\chi_L$ is 
  also $2$-divisible of cardinality $3$, so that we can apply Proposition~\ref{prop_simplex_repetition_special}. In the remaining part we can assume that each three points span a plane and 
  denote the dimension of the span of $\cM$ by $k$. Since each line contains at most two points of $\cM$, we have $k\ge 4$. If $k=4$, then there has to be a plane $E$ containing four points. 
  Since no three points of $\cM|_E$ form a line, there exists a line $L\le E$ such that $\cM|_E=\chi_E-\chi_L=\chi_{E\backslash L}$. However, then $\cM|_E$ is $2$-divisible,
  c.f.~Proposition~\ref{prop_2_div_card_4}, and $\cM-\cM|_E=2\cdot \chi_P$ for a suitable point $P$. Thus, it remains to consider the case $k=5$. Here $2$-divisibility implies $\cM(E)\le 3$ 
  for every plane $E$ since otherwise $\cM(\langle E,P\rangle)=6$ for each point $P\notin E$ with $\cM(P)=1$, which contradicts $k=5$. Moreover, no five points can span a solid, which is a 
  hyperplane for $k=5$.  
\end{proof}
We remark that in the case where $\cM$ contains a point $P$ of multiplicity at least two we can apply Proposition~\ref{prop_2_div_card_4} to $\cM-2\cdot\chi_P$.

\begin{proposition}
  \label{prop_2_div_card_7}
  Let $\cM$ be a $2$-divisible multiset of points in $\PG(v-1,2)$ with cardinality $7$. Then either $\cM$ contains a point of multiplicity at least two, there exists a line $L$ such that 
  $\cM-\chi_L\ge 0$, or $\cM$ is the characteristic function of a projective base $B_7$ of size $7$.
\end{proposition}
\begin{proof}
  W.l.o.g.\ we assume $\gamma_1(\cM)=1$ and $\cM(L)\le 2$ for every line $L$. If there would be a plane $E$ with $\cM(E)\ge 4$, then there would be a line $L\le E$ such that 
  $\cM|_E=\chi_{E\backslash L}$. However, in this case $\cM|_E$ is $2$-divisible and $\cM-\cM|_E$ would be $2$-divisible with cardinality $3$, which is possible for the characteristic function 
  of a line only. Thus, we assume $\cM(E)\le 3$ for every plane and that each three points of $\cM$ span a plane. So, we have $k\ge 5$ for the dimension of the span of $\cM$ since 
  the standard equations do not have a solution with $a_5=0$ for $k=4$. For $k=5$ there would be a solid $S$ with $\cM(S)=5$. Noting that $\cM(L)\le 2$ for each line $L\le S$ and 
  $\cM(E)\le 3$ for each plane $E\le S$ we conclude that $\cM|_S$ would be a projective base of size $5$, so that $\cM-\cM|_S=2\cdot \chi_P$ for a suitable point $P$, 
  which contradicts our assumption. Thus, it remains to consider the case $k=6$. Here we have $\cM(S)\le 4$ for each solid since otherwise 
  $\cM(\langle S,P\rangle)=7$ for each point $P\notin S$ with $\cM(P)=1$. From $2$-divisibility we conclude $\cM(H)\le 5$ for each hyperplane $H$, so that 
  $\cM$ has to be the characteristic function of a projective base of size $7$.  
\end{proof}
In the case where $\cM$ contains a point $P$ of multiplicity at least two we can apply Proposition~\ref{prop_2_div_card_5} to $\cM-2\cdot\chi_P$. If a line $L$ with 
$\cM-\chi_L\ge 0$ exists, then we can apply Proposition~\ref{prop_2_div_card_4} to $\cM-\chi_L$. We remark that for each dimension $3\le k\le 7$ of the span of $\cM$ 
there exists an up to symmetry unique example, if we assume $\gamma_1(\cM)=1$. For even sets of points over $\F_2$ of cardinality $n\ge 8$ the classification gets more 
involved, see \cite{ubt_eref40887} for computational results.\footnote{Note that adding a parity check bit to an arbitrary binary linear code yields a $2$-divisible linear code
whose effective length is increased by one, so that the classification of even sets of points over $\F_2$ is equivalent to the classification of sets of points over $\F_2$ (in some sense).}

\bigskip

Let $\cM$ be a doubly-even multiset of points over $\F_2$. Cardinalities $n\in\{4,6,7\}$ are characterized in Corollary~\ref{cor_simplex_repetition_special} and cardinality 
$n=8$ is characterized in Proposition~\ref{prop_multiples_of_affine_spaces}. Due to Proposition~\ref{prop_Gamma_infty} we have $\#\cM\ge 10$ for all other feasible cases. 

\begin{proposition}
  \label{prop_4_div_card_10}
  Let $\cM$ be a $4$-divisible multiset of points in $\PG(v-1,2)$ with cardinality $10$. Then there either exists a point $P$ with $\cM(P)\ge 4$ or $\cM=2\cdot \chi_{B_5}$ 
  where $B_5$ denotes a projective base of size $5$. 
\end{proposition}
\begin{proof}
  W.l.o.g.\ we assume $\gamma_1(\cM)\le 3$ and that $\cM$ is spanning with dimension $k$. From the standard equations we compute $a_2=2^{k-2}+1$ and $a_6=3\cdot 2^{k-2}-2$, so that 
  $\lambda_2\ge 1$ since each hyperplane with multiplicity $2$ is $2$-divisible, see Lemma~\ref{lemma_divisibility_hyperplane}, and so contains a double point $P_2$. If $P_3$ 
  is a point with multiplicity $3$, then each  hyperplane $H$ containing $P_3$ has multiplicity $\cM(H)=6$, so that Lemma~\ref{lemma_lower_bound_space_mult} yields $k=4$ via 
  $\cM(P_3)=2+2^{4-k}$. For the line $L$ spanned by $P_2$ and $P_3$ we have $\cM(H)=6$ for all hyperplanes $H$ containing $L$, so that Lemma~\ref{lemma_lower_bound_space_mult} 
  yields $\cM(L)=4<\cM(P_2)+\cM(P_3)$ -- contradiction. Thus, we have $\gamma_1(\cM)=2$ and $k\ge 3$ since $(2^2-1)\cdot 2=6<10$. Moreover, $k\neq 3$ since  
  $\cM'$ defined by $\cM'(P)=2-\cM(P)$ is also $4$-divisible with cardinality $2\cdot \left(2^{k}-1\right)-10=4$, which is impossible. 
  
  Let $L$ be a line with $\cM(L)\ge 3$, which clearly exists due to $\lambda_2\ge 1$ and $\lambda_1+\lambda_2\ge 2$. Each hyperplane $H$ containing $L$ 
  has multiplicity $\cM(H)=6$, so that Lemma~\ref{lemma_lower_bound_space_mult} yields $\cM(L)=4$ and $k=4$. With this, solving the standard equations gives 
  $\lambda_1=0$ and $\lambda_2=5$. Thus, $\tfrac{1}{2}\cdot\cM$ is $2$-divisible with cardinality $5$ and we can apply Proposition~\ref{prop_2_div_card_5}.
\end{proof}

Using Corollary~\ref{cor_simplex_repetition_special} and Proposition~\ref{prop_4_div_card_10} we conclude:
\begin{corollary}  
  \label{cor_4_div_card_10}
  Let $\cM$ be a $4$-divisible multiset of points in $\PG(v-1,2)$ with cardinality $10$. Then, we have $\cM(P)\in\{0,2,4,6\}$ for every point $P$.
\end{corollary}  
  
\begin{lemma}
  \label{lemma_lower_bound_point_mult_special1}
  Let $0\le l< r$ be integers and $\cM$ be a spanning $q^r$-divisible multiset of points in $\PG(k-1,q)$ of cardinality $n=q^l\cdot \frac{q^{r+1-l}-1}{q-1}+q^r$. Then $\gamma_1(\cM)=q^l$ or 
  $\gamma_1(\cM)=q^r-\frac{q^l-q^{2+r-k}}{q-1}$.
\end{lemma} 
\begin{proof}
  The possible hyperplane multiplicities are $m_1:=q^l\cdot \frac{q^{r+1-l}-1}{q-1}=n-q^r$ and $m_2:=q^l\cdot \frac{q^{r-l}-1}{q-1}=n-2q^r$. Due to Lemma~\ref{lemma_repeated_simplex} 
  both multiplicities indeed occur. If $H$ is a hyperplane with multiplicity $\cM(H)=m_2=q^l\cdot \frac{q^{r-l}-1}{q-1}$, then Proposition~\ref{prop_simplex_repetition_special} implies 
  the existence of an $(r-l)$-dimensional subspace $S'$ in $H$ with $\cM|_H=q^l\cdot \chi_{S'}$. Thus we have $\gamma_1(\cM)\ge q^l$. If $P$ is a point with multiplicity $\cM(P)>q^l$, 
  then each hyperplane through $P$ has cardinality $m_1$. Lemma~\ref{lemma_lower_bound_space_mult} yields 
  $$
    \cM(P)=n-\frac{q^r}{q^{k-2}}\cdot[k-1]_q=q^r-\frac{q^l-q^{2+r-k}}{q-1}.
  $$
\end{proof}

\begin{proposition}
  \label{prop_4_div_card_11}
  Let $\cM$ be a $4$-divisible multiset of points in $\PG(v-1,2)$ with cardinality $11$. Then, $\cM=\chi_E+4\cdot\chi_P$, where $E$ is a plane and $P$ a point.
\end{proposition}
\begin{proof}
  Since no $4$-divisible set of $11$ points in $\PG(v-1,2)$ exists, Lemma~\ref{lemma_lower_bound_point_mult_special1} implies $\gamma_1(\cM)=3+2^{2+r-k}\ge 4$, where
  $k$ is the dimension of the span of $\cM$ and $r=2$. Reducing the multiplicity of a point $P$ with maximum multiplicity by four gives a $4$-divisible multiset of points 
  with cardinality $7$, which is the characteristic function of a plane by Proposition~\ref{prop_simplex_repetition_special}. 
\end{proof}

\begin{proposition}
  \label{prop_4_div_card_12}
  Let $\cM$ be a $4$-divisible multiset of points in $\PG(v-1,2)$ with cardinality $12$. Then either there exists a point with multiplicity at least four 
  or all points have even multiplicity.
\end{proposition}
\begin{proof}
  W.l.o.g.\ we assume $\gamma_1(\cM)\le 3$. Assume that $P$ is a point with multiplicity $\cM(P)=3$. If $H$ is a hyperplane with $P\le H$ and $\cM(H)=4$, then we can apply 
  Proposition~\ref{prop_2_div_card_4} to conclude a contradiction since $\cM|_H$ is $2$-divisible, see Lemma~\ref{lemma_divisibility_hyperplane}. Thus, all hyperplanes
  containing $P$ have multiplicity $8$ and Lemma~\ref{lemma_lower_bound_space_mult} yields a contradiction. So, we conclude $\gamma_1(\cM)=2$ from Proposition~\ref{prop_Gamma_1}.
  
  Denoting the dimension of the span of $\cM$ by $k$ we conclude $k\ge 3$ from $2\cdot\left(2^k-1\right)\ge 11$. Moreover, $k\neq 3$ since $\cM'$ defined by 
  $\cM'(P)=2-\cM(P)$ is also $4$-divisible with cardinality $2\cdot 7-12=2$ otherwise, which is clearly impossible. Now assume that $P_1$ is a point with multiplicity $1$ 
  and $P_2$ a point with multiplicity $2$. Consider the line $L$ spanned by $P_1$ and $P_2$. Now observe that $\cM(H)\neq 4$ for each hyperplane $H$ containing $L$ 
  since $\cM|_H$ is $2$-divisible and Proposition~\ref{prop_2_div_card_4} would yield a contradiction otherwise. Since we have $\lambda_2\ge 1$, this implies $\lambda_1=0$.  
\end{proof}
So, we can read off the explicit classification from Proposition~\ref{prop_multiples_of_affine_spaces} or Proposition~\ref{prop_2_div_card_6}. 

\begin{proposition}
  \label{prop_div_4_card_13}
  Let $\cM$ be a $4$-divisible multiset of points in $\PG(v-1,2)$ with cardinality $13$. Then, either $\cM=\chi_E+2\cdot\chi_L$, where $E$ is a plane and $L$ a line, 
  or there exists a projective base $B_5$ of size $5$ and a point $C$ outside of the span of $B_5$ such that $\cM(C)=3$, $\cM(Q)=1$ if there exists a point $P$ in $B_5$ such that 
  $Q\in\langle P,C\rangle$, $Q\neq C$, and $\cM(Q)=0$ otherwise.
\end{proposition}
\begin{proof}
  Since no $4$-divisible multiset of points in $\PG(v-1,2)$ with cardinality $9$ exists, we have $\gamma_1(\cM)<4$ and $k\ge 3$ for the dimension $k$ of the span of $\cM$. 
  Due to Corollary~\ref{cor_simplex_repetition_special} it suffices to show the existence of a plane $E$ with $\cM\ge \chi_E$. If $k=3$, then we consider $\cM'$ defined by 
  $\cM'(P)=3-\cM(P)$. With this, $\cM'$ is also $4$-divisible, has cardinality $3\cdot \left(2^k-1\right)-\#\cM=8$ and maximum point multiplicity at most $3$. From 
  Proposition~\ref{prop_multiples_of_affine_spaces} we conclude $\gamma_1(\cM')\le 2$, so that we can choose $E$ as the ambient space. In the remaining part we have $k\ge 4$.
  Since there is no $2$-divisible multiset of cardinality $1$, we can use the standard equations to compute $a_5=5\cdot 2^{k-3}+1$, $a_9=3\cdot 2^{k-3}-2$, and $\lambda_2=1-3\lambda_3+2^{6-k}$.   
  
  Assume that $P_1,P_2$ are two different points with $\cM(P_1),\cM(P_2)\ge 2$. Let $L$ be the line spanned by $P_1,P_2$ and $H$ be an arbitrary hyperplane containing $L$. 
  Since $\cM|_H$ is $2$-divisible and contains both $P_1$ and $P_2$, Proposition~\ref{prop_2_div_card_5} yields $\cM(H)=9$. Applying Lemma~\ref{lemma_lower_bound_space_mult} 
  yields $\cM(L)=5+2^{5-k}$. If $k=5$, then $\lambda_2=3-3\lambda_3$ and the assumption $\lambda_2+\lambda_3\ge 2$ implies $\lambda_3=0$ and $\lambda_2=3$, so that $\cM\ge 2\cdot\chi_L$ (using $\cM(L)=6$). 
  If $k=4$, then we have $\cM(L)=7$ and $\lambda_3\le 1$ also implies $\cM\ge 2\cdot\chi_L$. Since $2\cdot \chi_L$ is $4$-divisible $\cM-2\cdot\chi_L$ is also $4$-divisible 
  with cardinality $7$, Corollary~\ref{cor_simplex_repetition_special} implies the existence of a plane $E$ with $\cM=2\cdot\chi_L+\chi_E$.
  
  If $\lambda_2+\lambda_3\le 1$, then $\lambda_2=1-3\lambda_3+2^{6-k}$ implies $\lambda_2=0$, $\lambda_3=1$, and $k=5$. If $P$ is a point with multiplicity $1$, then all 
  hyperplanes containing the line $L$ spanned by $P$ and the unique point of multiplicity $3$ have multiplicity $9$, so that Lemma~\ref{lemma_lower_bound_space_mult} yields $\cM(L)=5$. 
  In other words there exist five pairwise different lines $L_1,\dots,L_5$ that all contain the unique point of multiplicity $3$, denoted by $P_3$, such that $\cM=-2\chi_{P_3}+\sum_{i=1}^5 \chi_{L_i} $. 
  If $S$ is a solid, then $S$ can contain at most three of the lines $L_i$ (using $\cM(S)\le 9$) and intersects the others in a point. Thus, factoring out $P_3$ from the $L_i$ yields a 
  projective base of size $5$.      
\end{proof}

\begin{lemma}
  \label{lemma_aux_card_15_4_div_1}
  Let $\cM$ be a $4$-divisible multiset of points in $\PG(v-1,2)$ with $\#\cM=15$ and $2\le \gamma_1(\cM)\le 3$ that does not contain a plane in its support. Then, we have $k\ge 5$ 
  for the dimension $k$ of the span of $\cM$ and there does not exist a line $L'$ with $\cM(P)\ge 2$ for all points $P$ in $L'$. Moreover, we have 
   \begin{eqnarray}
    a_7 &=& 7\cdot 2^{k-3}+1-2a_3,\\ 
    a_{11} &=& 2^{k-3}-2+a_3,\text{ and}\\ 
    a_3 &=& \left(4+\lambda_2+3\lambda_3\right)\cdot 2^{k-6}-1 \label{card25_lambda}
  \end{eqnarray}   
  for the spectrum of $\cM$ and the multiplicity $\cM(S)$ of each subspace $S$ of codimension $2$ is odd.
\end{lemma}
\begin{proof}
  Since $\gamma_1(\cM)\le 3$, we have $k\ge 3$. If $k=3$, then we consider $\cM'$ defined by $\cM'(P)=3-\cM(P)$. With this, $\cM'$ is $4$-divisible with cardinality $6$, to that 
  Corollary~\ref{cor_simplex_repetition_special} implies $\gamma_1(\cM')\le 2$. Thus, we can choose $E$ as the ambient space and have $\cM\ge \chi_E$. In the remaining part we have $k\ge 4$. 
  The stated equations for the spectrum can be directly concluded from the standard equations. If $k=4$, then $\lambda_2+\lambda_3\ge 1$ implies $a_3\ge 1$. Let $H$ be a hyperplane with multiplicity 
  $3$ and $L$ be a line such that $\cM|_H=\chi_L$. For the two other hyperplanes $H'$, $H''$ that contain $L$ w.l.o.g.\ we can assume $\cM(H')=7$ and $\cM(H'')=11$. Since $\cM|_{H'}$ is not the 
  characteristic function of a plane there exists a point $P\le H'$ with $\cM(P)\ge 2$ so that $\cM|_{H'}-\chi_L-2\cdot \chi_P$ is a double point and we have $\lambda_2\ge 2$. Thus,
  $a_3\in\N$ implies $\lambda_2=4$, $\lambda_3=0$, and $\lambda_1=7$. However, there cannot be seven points of multiplicity $1$ and two points of multiplicity $2$ in $H''$.
  
  If $L'$ is a line with $\cM(P)\ge 2$ for all points $P$ in $L'$, then $\cM-2\cdot \chi_{L'}$ would be $4$-divisible with cardinality $9$, which is impossible. If $S$ is a subspace 
  of codimension $2$, then denote the three hyperplanes containing $S$ by $H_1,H_2,H_3$, so that $\#\cM+2\cdot\cM(S)=\cM(H_1)+\cM(H_2)+\cM(H_3)\equiv 1\pmod 4$ yielding 
  $\cM(S)\equiv 1\pmod 2$.  
\end{proof}

\begin{lemma}
  \label{lemma_aux_card_15_4_div_2}
  Let $\cM$ be a $4$-divisible multiset of points in $\PG(v-1,2)$ with $\#\cM=15$ and $2\le \gamma_1(\cM)\le 3$ that does not contain a plane in its support. There exist $a_3$ lines 
  $L_i$ sharing a common point $B$ such that $\cM(P)=1$ for all points $P$ contained in one of the lines $L_i$ and $\lambda_1\ge 2a_3+1$. Moreover, we have $k=5$ 
  for the dimension $k$ of the span of $\cM$.
\end{lemma}
\begin{proof}
  Let $H$ be a hyperplane of multiplicity $3$ and $L$ be a line with $\cM|_H=\chi_L$. If $H'$ is another hyperplane with multiplicity $3$ and $L'$ a line with $\cM|_{H'}=\chi_{L'}$, then we have 
  $L\neq L'$ since otherwise $\cM(H\cap H')=3$ and the third hyperplane containing $H\cap H'$ would have multiplicity $15=\#\cM$. So, there exist $a_3$ lines $L_i$ such that $\cM(P)=1$ for all points 
  $P$ contained in one of the lines $L_i$. Moreover any two such lines $L_i$ intersect in exactly a point.  
  
  So, if $a_3\le 2$, then there exist $a_3$ lines $L_i$ sharing a common point $B$ such that $\cM(P)=1$ for all points $P$ contained in one of the lines $L_i$ and $\lambda_1\ge 2a_3+1$.
  
  If there exist three of the lines $L_i$ with pairwise different intersection points, then they span a plane $E$ with six points of multiplicity $1$ noting that the $7$th point $P$ has multiplicity $0$ 
  since $\cM$ does not contain a plane in its support. Since the multiplicity of every subspace of codimension $2$ is odd, we have $k\ge 6$, see Lemma~\ref{lemma_aux_card_15_4_div_1}. With this 
  Equation~(\ref{card25_lambda}) yields $a_3\ge 4$. The fourth line $L_i$ also has to be completely contained in $E$, which is impossible. 
  
  Thus, in general all lines $L_i$ intersect in a common point $B$ and we conclude $\lambda_1\ge 2a_3+1$. Due to Lemma~\ref{lemma_aux_card_15_4_div_1} it remains to show $k\le 5$. From 
  $\lambda_2+\lambda_3\ge 1$ we conclude $\lambda_1\le 13$ and $a_3\le 6$, so that Equation~(\ref{card25_lambda}) implies $k\le 6$. If $k=6$, then Equation~(\ref{card25_lambda}) 
  gives $a_3\ge 4$. Let $S$ be a solid spanned by three of the lines $L_i$. If $S$ also contains a fourth line $L_i$, then we have $\cM(S)\ge 9$ and all three hyperplanes containing 
  $S$ have multiplicity $11$ and indeed $\cM(S)=9$. Let $H$ be one of these hyperplanes that contains a point $Q$ with multiplicity at least $2$, so that indeed $\cM(Q)=2$. W.l.o.g.\ 
  we assume that the four lines in $S$ are labeled $L_1,\dots,L_4$. Note that $\cM|_H$, $2\cdot \chi_Q$, $\chi_{L_1}$, and $\chi_{(L_2\cup L_3)\backslash B}$ are $2$-divisible so that 
  $\cM|_H-2\cdot\chi_Q-\chi_{L_1}-\chi_{(L_2\cup L_3)\backslash B}=\chi_{L_4\backslash B}$ is 
  also $2$-divisible, which is a contradiction. Thus, any four lines $L_i$ span a hyperplane $H$ with multiplicity $11$. Pick one of these and a solid $S$ that intersects these four $L_i$ in 
  exactly point $B$, so that $\cM(S)\in\{1,3\}$. If $\cM(S)=1$, then the two other hyperplanes that contain $S$ have multiplicity $3$, which would imply $\gamma_1(\cM)=1<2$. If 
  $\cM(S)=3$, then one of the other two hyperplanes containing $S$ has multiplicity $3$ and one has multiplicity $7$, which implies $a_3\ge 5$. With this, Equation~(\ref{card25_lambda}) 
  gives $a_3=5$, $\lambda_1=11$, $\lambda_2=2$, and $\lambda_3=0$. Let $\mathcal{H}$ denote the $16$ hyperplanes not containing $B$. Since $5\cdot 1+2\cdot 2=9<11$ we have 
  $\cM(H)\in \{3,7\}$ for all $H\in \mathcal{H}$, so that $8\cdot 14=\sum_{H\in\mathcal{H}} (\cM(H)-1) \ge 6\cdot 16$ -- contradiction.
  
  Thus, we have $k\le 5$ and Lemma~\ref{lemma_aux_card_15_4_div_1} gives $k=5$.               
\end{proof}

\begin{proposition}
  \label{prop_4_div_card_15}
  Let $\cM$ be a $4$-divisible multiset of points in $\PG(v-1,2)$ with cardinality $15$. Then either there exists a point of multiplicity at least $4$, $\gamma_1(\cM)=1$, or there 
  exists a plane $E$ with $\cM\ge\chi_E$.
\end{proposition}
\begin{proof}
  W.l.o.g.\ we assume $2\le \gamma_1(\cM)\le 3$ and that $\cM$ does not contain a plane in its support. Lemma~\ref{lemma_aux_card_15_4_div_2} states $k=5$ for the dimension of the 
  span of $\cM$, so that Equation~(\ref{card25_lambda}) yields $a_3\ge 2$. Using the notation from Lemma~\ref{lemma_aux_card_15_4_div_2} we consider the plane $E:=\left\langle L_1,L_2\right\rangle$ 
  with $\cM(E)\ge 5$. Due to Lemma~\ref{lemma_aux_card_15_4_div_1} $\cM(E)$ is odd. Since $\cM$ does not contain a plane in its support, we have $\cM(E)\in\{5,7\}$. If $\cM(E)=7$, then 
  there exists a point $P\le E$ with multiplicity $\cM(P)=2$. However, $E$ is contained in a hyperplane $H$ of multiplicity $7$ and $\cM|_H-2\cdot \chi_P-\chi_{L_1}=\chi_{L_2\backslash B}$ 
  is $2$-divisible -- contradiction. Thus, $\cM(E)=5$ and for the three hyperplanes $H_1,H_2,H_3$ containing $E$ we can assume w.l.o.g.\ $\cM(H_1)=7$, $\cM(H_2)=7$, and $\cM(H_3)=11$.
  Since $\cM|_{H_i}-\chi_{L_1}$ is $2$-divisible for $i=1,2$ we have $\gamma_1(H_i)=1$. Now let $P\le H_3$ be a point of multiplicity at least $2$, so that $\cM|_{H_3}-2\cdot\chi_P-\chi_{L_1}$ 
  is $2$-divisible with cardinality $6$, so that Proposition~\ref{prop_2_div_card_6} implies $\lambda_2+\lambda_3\ge 2$. With this we conclude $\lambda_1=11$, $\lambda_2=2$, $\lambda_3=0$, and 
  $\left(a_3,a_7,a_{11}\right)=(2,25,4)$. For two hyperplanes $H$, $H'$ of multiplicity $11$ let $E'$ denote their intersection. Counting points gives $\cM(E')\ge 7$. Since $\cM$ does not contain 
  a plane in its support, we have $\cM(E')=7$, so that the third hyperplane $H''$ containing $E'$ has multiplicity $7$ and contains a double point $P$. So, $\cM|_{H''}-2\cdot \chi_P$ is 
  $2$-divisible with cardinality $5$ and dimension at most $3$. Using Proposition~\ref{prop_2_div_card_5} we conclude that $E'$ contains both points of multiplicity $2$ and three points of 
  multiplicity $1$ that form a line. Let $L'$ be the line spanned by the two double points and $Q$ be the third point on the line, so that $\cM(Q)=1$ and $\cM(L')=5$. Each of ${4\choose 2}=6$ 
  pairs of hyperplanes of multiplicity $11$ yields a different plane $E'\ge L'$, so that $\lambda_1\ge 1+6\cdot 2=13$, which is a contradiction. 
\end{proof}
We remark that if a $4$-divisible multiset of points $\cM$ in $\PG(v-1,2)$ with cardinality $15$ contains a point $P$ with multiplicity at least $4$, then 
$\cM-4\cdot\chi_P$ is also $4$-divisible with cardinality $11$, so that Proposition~\ref{prop_4_div_card_11} implies the existence of a plane in the support of $\cM$. For 
$\gamma_1(\cM)=1$ the possibilities have been classified in \cite{honold2019classification}. Except for a single case all point sets also contain a plane in its support. We summarize 
the result in:
\begin{corollary} 
  \label{cor_4_div_card_15}
  Let $\cM$ be a $4$-divisible multiset of points in $\PG(v-1,2)$ with cardinality $15$. Then either there exists a plane $E$ and a $4$-divisible multiset of points $\cM'$ in $\PG(v-1,2)$ 
  with cardinality $8$ such that $\cM=\chi_E+\cM'$ or there exists a projective base $B_7$ of seven points and a point $P$ outside of $\langle B_7\rangle$ such that 
  $\cM(Q)=1$ iff there is a point $P'$ in $B_7$ with $Q\le \left\langle P',P\right\rangle$ and $\cM(Q)=0$ otherwise.
\end{corollary}

\begin{lemma}
  \label{lemma_4_div_card_16_special}
  Let $\cM$ be a $4$-divisible multiset of points in $\PG(v-1,2)$ with cardinality $16$ and $\gamma_1(\cM)\le 3$. Then, we have $\lambda_3\le 4$ and $\lambda_2+3\lambda_3\le 12$. 
  If $\lambda_3<4$, then $\lambda_3\le 2$.
\end{lemma}
\begin{proof}
  W.l.o.g.\ we assume $\gamma_1(\cM)=3$. Note that $\lambda_2\le\tfrac{16-3\lambda_3}{2}$ implies $\lambda_2+3\lambda_3\le 8+\tfrac{3\lambda_3}{2}\le 12.5$ for $\lambda_3\le 3$. 
  So, let $P_1,P_2,P_3$ be three arbitrary different points with multiplicity at least $2$. If they form a line $L$, then $\cM-2\cdot\chi_L$ is $4$-divisible with cardinality $10$, 
  so that Corollary~\ref{cor_4_div_card_10} yields a contradiction. Thus, any three points with multiplicity at least two span a plane $E$. If $E$ contains a fourth point of multiplicity 
  at least $2$, then they form an affine plane $A$, so that we can apply Proposition~\ref{prop_multiples_of_affine_spaces} to conclude the statement. Since there is no $4$-divisible multiset 
  of points with cardinality $[3]_2\cdot 3-\#\cM=7\cdot 3-16=5$, we have $k\ge 4$ for the dimension of the span of $\cM$. For the case $k=4$ consider a plane $E$ spanned by three 
  points of multiplicity $3$. Since $\cM(E)\ge 9$, we have $\cM(E)=12$. However, the three different lines spanned by the three pairs of the considered three points of multiplicity $3$ 
  have even cardinality and the third point has multiplicity at most $1$, so that $E$ contains three points with multiplicity $0$ that form a line $L'$. Since the fourth point in $E\backslash L'$ 
  has multiplicity at most $1$, we obtain the contradiction $\cM(E)\le 10$. Thus, we can assume $k\ge 5$ in the following.  

  Let $E$ be a plane spanned by three points of multiplicity $3$. Since $\cM(E)\ge 9$ we have $\cM(H)=12$ for every hyperplane $H$ containing $E$, so that 
  Lemma~\ref{lemma_lower_bound_space_mult} yields $\cM(E)=8+2^{6-k}$. An arbitrary line $L$ that is contained in hyperplanes of multiplicity $12$ only has multiplicity 
  $\cM(L)=8+2^{5-k}$, which is impossible. Now assume $k=5$ for a moment, so that $\cM(E)=10$ and $E$ contains an additional point $Q$ with multiplicity $1$. 
  If there would be a line $L\le E$ with multiplicity $7$, then for a hyperplane $H\ge L$ with multiplicity $8$ we would have that $\cM|_H-\chi_L$ is $2$-divisible with 
  cardinality $5$ containing two double points -- contradiction. Thus, the four points with non-zero multiplicity in $E$ form an affine plane, i.e., $\cM|_E$ is $2$-divisible.
  So, for each of the three hyperplanes $H$ containing $E$ we have $\cM(H)=12$ and $\cM|_H-\cM|_E$ is $2$-divisible with cardinality $2$, i.e., a double point. Thus, we conclude
  $\lambda_3=3$, $\lambda_2=3$, and $\lambda_1=1$. However, for a line $L$ spanned by two points $P',P''$ of multiplicity $3$ we have $\cM(L)=6$ there exists a hyperplane $H\ge L$ 
  with multiplicity $\cM(H)=8$. Thus, $\cM|_H-2\cdot\chi_{P'}-2\cdot\chi_{P''}$ is $2$-divisible with cardinality $4$ containing at least two points of multiplicity $2$, so
  that Proposition~\ref{prop_multiples_of_affine_spaces} yields the existence of two points with multiplicity $1$ outside of $L$. This contradicts $\lambda_1=1$ and it remains to consider 
  the case $k=6$. Here the plane $E$ spanned by three points of multiplicity $3$ has multiplicity $9$ and any solid containing $S$ has multiplicity $10$. Thus, we have 
  $\lambda_3=3$, $\lambda_2=0$, and $\lambda_1=7$. Now consider a line $L$ spanned by two points of multiplicity $3$, so that $\cM(L)=6$, and let $H\ge L$ be a hyperplane with multiplicity
  $8$, so that Proposition~\ref{prop_multiples_of_affine_spaces} yields that $\cM|_H$ spans a plane $E'$ and the four points of non-zero multiplicity form an affine plane. Let $S$ be 
  the solid spanned by $E'$ and the third point of multiplicity $3$, so that $\cM(S)\ge 8+3=11$, which is impossible for $k=6$.   
\end{proof}

\begin{lemma}
  \label{lemma_4_div_card_16_special2}
  Let $\cM$ be a $4$-divisible spanning multiset of points in $\PG(k-1,2)$ with cardinality $16$, $\gamma_1(\cM)=3$, and $\lambda_2\ge 2$. Then, we have 
  $\left(\lambda_1,\lambda_2,\lambda_3\right)\in\big\{(7,3,1),(6,2,2),(9,2,1)\big\}$.
\end{lemma}
\begin{proof}
  Let $P_1,P_2,P_3$ be three arbitrary different points with multiplicity at least $2$. If they form a line $L$, then $\cM-2\cdot\chi_L$ is $4$-divisible with cardinality $10$, 
  so that Corollary~\ref{cor_4_div_card_10} yields a contradiction. Thus, any three points with multiplicity at least two span a plane $E$. If $E$ contains a fourth point of multiplicity 
  at least $2$, then they form an affine plane $A$, so that we can apply Proposition~\ref{prop_multiples_of_affine_spaces} to conclude the statement. Assume $\cM(P_1)=3$,  
  $\cM(P_2)=\cM(P_3)=2$, and that each plane contains at most three points of multiplicity at least $2$ in the following. Since there is no $4$-divisible multiset of cardinality $9$, 
  $E$ contains at least one point with multiplicity $0$, so that $7\le \cM(E)\le 10$. Since $\cM|_H$ is $2$-divisible for any hyperplane $H$, we have $\cM(H)=12$ if $H\ge E$, so that 
  Lemma~\ref{lemma_lower_bound_space_mult} yields $\cM(E)=8+2^{6-k}$, which implies 
  $\cM(E)\in \{9,10\}$ and $k\in\{5,6\}$. Since $E$ contains only three points of multiplicity at least $2$, we have $a_0=0$. With this, we conclude 
  $\lambda_2+3\lambda_3=2^{6-k}\cdot\left(3+a_{12}\right)-8$. Let $Q$ be an arbitrary point outside of $E$ and $S$ be the solid spanned by $Q$ and $E$. If $k=5$, then $\cM(S)=12$ and 
  $\cM(E)=10$ implies $\cM(Q)\le 2$, so that $\lambda_3=1$. Using $\lambda_2+3\lambda_3=2a_{12}-2$ we conclude $\lambda_2\equiv 1\pmod 2$, so that $\lambda_2\ge 3$. Now consider a 
  plane $E'$ spanned by three points of multiplicity $2$. Since $E'$ does not contain a fourth point of multiplicity at least $2$, we conclude that every hyperplane $H'\ge E'$ has 
  multiplicity $12$. As before, we can conclude $\cM(E')=10$, which then implies $\cM(P)\ge 1$ for all $P\le E$ -- contradiction. It remains to consider the case $k=6$ where 
  $\cM(E)=9$ and $\cM(S)=10$ for every solid $S\ge E$, so that $\lambda_2+\lambda_3=3$.    
\end{proof}

 \begin{lemma}
  \label{lemma_4_div_card_17_aux}
  Let $\cM$ be a $4$-divisible multiset of points in $\PG(v-1,2)$ with cardinality $17$ and $2\le \gamma_1(\cM)\le 3$. Then, we have  
  \begin{eqnarray*}
    a_5 &=& 2^{k-3}+2+a_{13},\\\ 
    a_9 &=& 7\cdot 2^{k-3}-3-2a_{13},\\
    \lambda_2+3\lambda_3 &=& -5+2^{6-k}\cdot\left(3+a_{13}\right),
  \end{eqnarray*}
  and $k\ge 4$. Moreover, each three points of multiplicity at least $2$ span a plane and each four points of multiplicity at least $2$ span a solid.
\end{lemma} 
\begin{proof}
  Since no $2$-divisible multiset of points of cardinality $1$ exists, the multiplicities of the hyperplanes are contained in $\{5,9,13\}$. From the standard equations we compute the 
  stated equations. Clearly we have $k\ge 3$. If $k=3$, then $\cM'$ defined by $\cM'(P)=3-\cM(P)$ would be $4$-divisible with $\#\cM'=4$ and $\gamma_1\!\left(\cM'\right)\le 3$ -- contradiction. 
  If $L'$ is a line with $\cM\ge 2\cdot\chi_{L'}$, then $\cM-2\cdot\chi_{L'}$ would be $4$-divisible with cardinality $11$, so that Proposition~\ref{prop_4_div_card_11}
  implies $\gamma_1(\cM)\ge 4$ -- contradiction. In other words each three points of multiplicity at least $2$ span a plane. If $E$ is a plane and $L\le E$ a line with
  $\cM\ge 2\cdot \chi_{E\backslash L}$, then $\cM-2\cdot \chi_{E\backslash L}$ would be $4$-divisible with cardinality $9$ -- contradiction. Thus, any four points of multiplicity 
  at least $2$ span a solid. 
\end{proof}

\begin{lemma}
  \label{lemma_4_div_card_17_special}
  Let $\cM$ be a $4$-divisible multiset of points in $\PG(v-1,2)$ with cardinality $17$ and $2\le \gamma_1(\cM)\le 3$. Then there exists a point $P$ with $\cM(P)=2$, $k\ge 6$, 
  or $\lambda_3=1$.
\end{lemma}
\begin{proof}
  Assuming $\lambda_2=0$ the standard equations yield
  \begin{eqnarray*}
    a_5 &=& 2^{k-3}+2+a_{13},\\ 
    a_9 &=& 7\cdot 2^{k-3}-3-2a_{13},\text{ and}\\
    3\lambda_3 &=& \left(3+a_{13}\right)\cdot 2^{6-k}-5,
  \end{eqnarray*}
  where $k$ denotes the dimension of the span of $\cM$. If $k\le 3$, then $\lambda_3\le 5$ implies $a_{13}<0$  -- contradiction. If $k=4$, then $\lambda_3\equiv 1\pmod 4$, so that we can assume $\lambda_3=5$. With this we have $a_5=6$, $a_9=7$, $a_{13}=2$, and $\lambda_1=2$. Let $L$ be the line spanned by the two points of multiplicity one. From Proposition~\ref{prop_2_div_card_5} we conclude that every hyperplane of multiplicity $5$ has to contain $L$. However, $L$ is contained in three hyperplanes only -- contradiction.

  Finally, assume $k=5$ and $\lambda_3\ge 2$.
  If $\lambda_3=2$, then $3\lambda_3 = \left(3+a_{13}\right)\cdot 2^{6-k}-5$ would imply that $a_{13}$ is fractional. So, let $P_1,P_2,P_3$ be different points with multiplicity $3$. 
  They cannot form a line $L$ since $\cM-2\cdot\chi_L$ would be $4$-divisible with cardinality $11$ but does not 
  contain a point of multiplicity at least $4$, which contradicts Proposition~\ref{prop_4_div_card_11}. So, let $E$ be the plane spanned by $P_1$, $P_2$, and $P_3$. Clearly every hyperplane 
  $H$ containing $E$ has multiplicity at least $9$. However, multiplicity $9$ is impossible, since otherwise $\cM|_H-\sum_{i=1}^3 2\cdot \chi_{P_i}$ would be $2$-divisible of cardinality $3$, so that 
  Corollary~\ref{cor_simplex_repetition_special} yields that $P_1$, $P_2$, $P_3$ form a line -- contradiction. Thus, we have $\cM(E)=11$, which implies $\lambda_3=3$, $\lambda_1=8$, $a_{13}=4$, 
  $a_5=10$, and $a_9=17$. Now let $L\le E$ be a line with multiplicity $7$, so that $\cM|_L$ is $2$-divisible and all $7$ hyperplanes containing $L$ have multiplicity $13$ -- contradiction.
\end{proof}

\begin{lemma}
  \label{lemma_4_div_card_17_special_2}
  Let $\cM$ be a $4$-divisible multiset of points in $\PG(v-1,2)$ with cardinality $17$ and $2\le \gamma_1(\cM)\le 3$. If there exists a line $L$ consisting of three 
  points $P_i$ with $\cM\!\left(P_i\right)=i$ for $1\le i\le 3$, then we have $k=5$, $\lambda_1=9$, $\lambda_2=1$, $\lambda_3=2$, $a_{13}=3$, $a_9=19$, and $a_5=9$. 
  Up to symmetry a unique representation of $\cM$ is given by the columns of 
  $$
    \begin{pmatrix}
    111 & 1 & 111 & 1 & 00 & 0 & 00 & 00 & 00 \\
    000 & 1 & 111 & 0 & 11 & 0 & 00 & 00 & 11 \\
    000 & 0 & 111 & 1 & 00 & 1 & 01 & 01 & 01 \\
    000 & 0 & 000 & 0 & 00 & 0 & 11 & 00 & 11 \\
    000 & 0 & 000 & 0 & 00 & 0 & 00 & 11 & 11 \\ 
    \end{pmatrix}.
  $$ 
\end{lemma}
\begin{proof}
  Since $\cM(L)=6\not\equiv \#\cM \pmod 2$, Lemma~\ref{lemma_divisibility_hyperplane} implies $k\ge 5$. Assume that $Q$ is another point with $\cM(Q)\ge 2$ and consider the
  plane $E$ spanned by $L$ and $Q$. If $H\ge E$ is a hyperplane with multiplicity $9$, then $\cM':=\cM|_H-\chi_L-2\cdot \chi_{P_3}-2\cdot\chi_Q$ would be $2$-divisible with cardinality 
  $2$, so that
  $\cM'=2\cdot\chi_P$ for some point $P$. However, we have $\cM'\!\left(P_2\right)=1$ -- contradiction. Thus, every
  hyperplane $H\ge E$ has multiplicity $13$, so that Lemma~\ref{lemma_lower_bound_space_mult} yields $\cM(E)=9+2^{6-k}$. 
  
  If $k=5$, then $\cM(E)=11$ and the other six planes $E'$ containing $L$ have multiplicity $7$, so that all points of multiplicity at least $2$ are contained in $E$. Lemma~\ref{lemma_4_div_card_17_aux} then yields $\lambda_2+\lambda_3=3$,
  so that $\lambda_2+3\lambda_3 = -5+2^{6-k}\cdot\left(3+a_{13}\right)$ and $a_{13}\ge 3$ implies $\lambda_2=1$, $\lambda_3=2$, $\lambda_1=9$, and $a_{13}=3$. Using the equations 
  in Lemma~\ref{lemma_4_div_card_17_aux} we then compute $a_9=19$ and $a_5=9$. 
  
  If $k=6$, then $\cM(E)=10$ and each solid $S\ge E$ has multiplicity $11$. Thus, as before, we conclude that all points with multiplicity at least $2$ are contained in $E$ and 
  $\lambda_2+\lambda_3=3$. Let $R_1,R_2,R_3$ be pairwise different points such that $\cM|_E-\chi_L-2\cdot\chi_{P_3}-2\cdot\chi_Q=\sum_{i=1}^3\chi_{R_i}$. For any two points $Z_1,Z_2$ 
  of multiplicity $1$ outside of $E$ there exists a third point $Z_3$ (depending on $Z_1$ and $Z_2$), so that $\cM|_H=\cM|_E+\sum_{i=1}^3 \chi_{Z_i}$ for the hyperplane 
  $H=\left\langle E,Z_1,Z_2\right\rangle$ and $\cM':=\sum_{i=1}^3 \chi_{R_i}\,+\,\sum_{i=1}^3 \chi_{Z_i}$ is $2$-divisible consisting of six different points of multiplicity $1$.   
  Now we apply Proposition~\ref{prop_2_div_card_6}. If $\cM'$ is the sum of the characteristic functions of two disjoint lines, then one of the two lines has to be contained in $E$, 
  i.e., the $R_i$ form a line. In that case, also the $Z_i$ have to form a line for each choice of a pair $\left\{Z_1,Z_2\right\}$, so that the seven points of multiplicity $1$ outside of 
  $E$ form a disjoint plane $E'$. However, then we can apply Proposition~\ref{prop_4_div_card_10} to $\cM-\chi_{E'}=\cM|_E$ to obtain a contradiction. Thus, $\cM'$ is a projective base of size 
  $6$ in all cases. This is impossible as it can be seen using coordinate representations: W.l.o.g.\ we assume $R_i=e_i$ for $1\le i\le 3$, where $e_j$ denotes the $j$th unit vector. Since $\cM$ 
  is spanning, we assume w.l.o.g.\ that also $e_i$ are points of multiplicity $1$ for $4\le i\le 6$.  Choosing $4\le i<j\le 6$ the points $e_1,e_2,e_3,e_i,e_j$ are 
  completed by $e_1+e_2+e_3+e_i+e_j$ to a projective base of size $6$. However, $e_1,e_2,e_3,e_1+e_2+e_3+e_i+e_j,e_1+e_2+e_3+e_i+e_h$ are completed by $e_j+e_h$ to a 
  projective base of size $6$ for each $\{i,j,h\}=\{4,5,6\}$, which gives more than seven points of multiplicity $1$ outside of $E$ -- contradiction.

  
  \medskip
  
  It remains to consider the case $\lambda_3=\lambda_2=1$ and $\lambda_1=12$. 
  The projection $\cM_{P_2}$ of $\cM$ through $P_2$ is $4$-divisible with cardinality $15$ and a unique point $P'$ of cardinality $4$ (arising from $P_1$ and $P_3$). However, 
  Proposition~\ref{prop_4_div_card_11} yields a contradiction for $\cM_{P_2}-4\cdot\chi_{P'}$.    
  
  \medskip

  For the classification of the case $k=5$, $\lambda_1=9$, $\lambda_2=1$, $\lambda_3=2$, $a_{13}=3$, $a_9=19$, and $a_5=9$, we consider a projection $\cM_{P_2}$ of 
  $\cM$ through the unique point $P_2$ of multiplicity $2$. Note that the line $L$ is mapped to a point $\overline{Q}$ with multiplicity $\cM_{P_2}(\overline{Q})=
  \cM\!\left(P_3\right)+\cM\!\left(P_1\right)=4$, so that $\cM_{P_2}-4\cdot\chi_{\overline{Q}}$ is $4$-divisible with cardinality $11$. From Proposition~\ref{prop_4_div_card_11} 
  we conclude the existence of a point $\overline{Q}'$ and a plane $\overline{E}'$ such that $\cM_{P_2}=4\cdot\chi_{\overline{Q}}+ 4\cdot\chi_{\overline{Q}'}+\chi_{\overline{E}'}$. 
  The preimage of $\overline{Q}'$ has to be a line $L'$ consisting of $P_2$, the second point of multiplicity $3$, that we denote by $P_3'$, and a point $P_1'$ of multiplicity $1$. 
  Since $P_2$ and $P_3'$ are contained in $E$, also $P_1'$ is contained in $E$. Let $\tilde{L}:=\left\langle P_1,P_1'\right\rangle$, $N$ be the third point in $\tilde{L}$, and 
  $A$ be the seventh point in $E$, i.e., the set of points in $E$ is given by $\left\{P_1,P_2,P_3,P_1',P_3',A,N\right\}$. Since $\cM(E)=11$ we have $\cM(A)+\cM(N)=1$. Each of the 
  three hyperplanes $H_1,H_2,H_3$ containing $E$ has multiplicity $13$ and consists of two points of multiplicity $1$ outside of $E$. By $L_i$ we denote the line 
  spanned by those two points, where $1\le i\le 3$. Since $k=5$, the line $L_i$ meets $E$ in a point $Z_i$. We will now determine $Z_i$ and show $Z_1=Z_2=Z_3$. Let $Z_i\le \hat{L}\le E$ be an
  arbitrary line and $\hat{E}=\left\langle \hat{L},L_i\right\rangle$ be a plane depending on the choice of $\hat{L}$. Note that $\cM(\hat{E})=\cM(\hat{L})+2$ and 
  $\cM(\hat{E})\equiv \#\cM\equiv 1\pmod 2$ due to Lemma~\ref{lemma_divisibility_hyperplane}. Thus $Z_i$ cannot be contained in $L$ or $L'$, which implies $Z_i\in\{A,N\}$. Noting 
  that $N\le \left\langle P_3,P_3'\right\rangle$ we then conclude $Z_i=A$ for all $1\le i\le 3$, which then implies $\cM(A)=1$ and $\cM(N)=0$. For a 
  parameterization of $\cM$ we denote the $i$th vector by $e_i$ and choose w.l.o.g.\ $A=e_3$, $P_2=e_2$, $P_3=e_1$, so that $P_3'=e_1+e_2+e_3$, $P_1=e_1+e_3$, $P_1'=e_1+e_2$, 
  and $N=e_2+e_3$. W.l.o.g.\ we choose $L_1=\left\langle e_3,e_4\right\rangle$ and $L_2=\left\langle e_3,e_5\right\rangle$. From $4$-divisibility we then conclude $L_3=\left\langle
  e_3,e_2+e_4+e_5\right\rangle$, which can be seen e.g.\ by looking at the first $15$ columns of the stated matrix and using the fact that the number of ones in each row 
  has to be divisible by $4$.         
\end{proof}

\begin{lemma}
  \label{lemma_4_div_card_17_special_3}
  Let $\cM$ be a spanning $4$-divisible multiset of points in $\PG(k-1,2)$ with cardinality $17$, $\lambda_2=0$, and $\gamma_1(\cM)=3$. Then, we have $\lambda_3\in\{1,2\}$ and 
  a line spanned by two points of multiplicity $3$ has multiplicity $6$. If $\lambda_3=2$, then $k\ge 6$.  
\end{lemma}
\begin{proof}
  Assume that $L$ is a line consisting of two points of multiplicity $3$ and one point of multiplicity $1$. Noting that $\cM|_L$ is $2$-divisible and $\lambda_2=0$ we 
  conclude $\cM(H)=13$ for each hyperplane $H\ge L$. Since $\cM(L)=7$, we have $k\ge 4$ and can use Lemma~\ref{lemma_lower_bound_space_mult} to conclude 
  $\cM(L)>9$ -- contradiction.
  
  Due to Lemma~\ref{lemma_4_div_card_17_special} it suffices to consider the case $\lambda_3\ge 3$. W.l.o.g. we assume that the points with the coordinates 
  $e_1$, $e_2$, and $e_3$ have multiplicity $3$, so that the points in $\left\langle e_1+e_2,e_1+e_3\right\rangle$ have multiplicity zero. Set $E:=\left\langle e_1,e_2,e_3\right\rangle$ 
  and $Q:=\left\langle e_1+e_2+e_3\right\rangle$. Since no $4$-divisible multiset of points of cardinality $9$ exists, we have $\cM(Q)\in \{0,1\}$ so that $\cM(E)\in \{9,10\}$. Note
  that $\cM|_E$ is not $2$-divisible if $\cM(E)=9$, so that all  hyperplanes $H\ge E$ have multiplicity $13$ and Lemma~\ref{lemma_lower_bound_space_mult} yields 
  $\cM(E)=9+2^{6-k}$, so that $k=6$ and $\cM(E)=10$, i.e., $\cM(Q)=1$. So, $\cM|_E$ is $2$-divisible and each hyperplane $H\ge E$ has multiplicity $13$ and is 
  given by $\cM|_H-\cM|_E=\chi_{L_H}$ for some line $L_H$. For an arbitrary point $P \not\in E$ with
  non-zero multiplicity consider the solid $S=\langle E,P\rangle$ and the three hyperplanes $H_1,H_2,H_3$ containing $S$. Since $\cM(H_1)=\cM(H_2)=\cM(H_3)=13$, we have $\cM(S)=11$
  and $\cM=\cM|_E+ \sum_{i=1}^3 \chi_{L_{H_i}} -2\chi_P$. Using the same argument for a
  different point $P\neq Q \not\in E$ with non-zero multiplicity yields a contradiction.
\end{proof}

\section{The minimum possible point multiplicity for triply-even multisets of points}
\label{sec_8_div}

Let $\cM$ be an $8$-divisible multiset of points in $\PG(v-1,2)$. From Corollary~\ref{cor_simplex_repetition_special} we know $\gamma_1(\cM)=8$ if $\#\cM=8$ and $\gamma_1(\cM)=4$ if 
$\#\cM=12$.

\begin{lemma}
  \label{lemma_8_div_card_22}
  Let $\cM$ be an $8$-divisible multiset of points in $\PG(v-1,2)$ with cardinality $22$. Then, we have $\gamma_1(\cM)\ge 8$.
\end{lemma}
\begin{proof}
  The possible hyperplane multiplicities are given by $14$ and $6$. If $H$ is a hyperplane with $\cM(H) = 6$, then 
  Corollary~\ref{cor_simplex_repetition_special} yields $\cM|_H=2\cdot\chi_L$  for some line $L$. If $P$ is a 
  point with multiplicity $\cM(P)> 2$, then all hyperplanes containing $P$ have multiplicity $14$ and 
  Lemma~\ref{lemma_lower_bound_space_mult} yields $\cM(P)=6+2^{5-k}$, so that it suffices to consider the case 
  $\cM(P)=7$. Since $\cM|_H-4\cdot \chi_P$ is $4$-divisible with cardinality $10$ and a point of multiplicity $3$, 
  Corollary~\ref{cor_4_div_card_10} gives a contradiction.
  
  It remains to consider the case $\gamma_1(\cM) = 1$. Using the standard equations we compute $\lambda_2=19+2^{8-k}$ for the 
  dimension $k$ of the span of $\cM$, which clearly contradicts $\#\cM=22$. 
\end{proof}

\begin{lemma}
  \label{lemma_8_div_card_23}
  Let $\cM$ be an $8$-divisible multiset of points in $\PG(v-1,2)$ with cardinality $23$. Then, we have $\gamma_1(\cM)\ge 8$.
\end{lemma}
\begin{proof}
  The possible hyperplane multiplicities are given by $15$ and $7$. If $H$ is a hyperplane with $\cM(H) = 7$, then 
  Corollary~\ref{cor_simplex_repetition_special} yields $\cM|_H=\chi_E$  for some plane $E$. If $P$ is a 
  point with multiplicity $\cM(P)> 1$, then all hyperplanes containing $P$ have multiplicity $15$ and 
  Lemma~\ref{lemma_lower_bound_space_mult} yields $\cM(P)>7$. However, not all hyperplanes can have multiplicity $15$.  
\end{proof}

\begin{lemma} 
  \label{lemma_8_div_card_37_aux_1}
  Let $\cM$ be an $8$-divisible multiset of points in $\PG(v-1,2)$ with cardinality $37$ and $\gamma_1(\cM)\le 7$. Then, we have 
  \begin{eqnarray}
    a_{13}&=&5\cdot 2^{k-4}+2+a_{29},\nonumber\\
    a_{21}&=& 11\cdot 2^{k-4}-3-2a_{29},\text{ and}\nonumber\\ 
    \sum_{i=2}^7 {i\choose 2} \lambda_i &=& 9+2^{8-k}\cdot \left(3+a_{29}\right) \label{eq_lambda_sum_8_div_card_37} 
  \end{eqnarray}
  for the spectrum of $\cM$, where $k$ is the dimension of the span of $\cM$. There do not exist a solid $S'$, a plane $E'$, or a line $L'$ 
  such that $\cM\ge \chi_{S'}$, $\cM\ge 2\cdot \chi_{E'}$, or $\cM\ge 4\cdot \chi_{L'}$, respectively. 
  Moreover, we have $\gamma_1(\cM)\le 3$, each point of multiplicity at most $3$ is contained in a hyperplane with multiplicity $13$, and $k\ge 6$.  
\end{lemma}
\begin{proof}
  Since there is no $4$-divisible multiset of cardinality $5$, the hyperplane multiplicities are given by $13$, $21$, and $29$. 
  With this, the stated equations follow from the standard equations. If a subspace $S'$, $E'$, or 
  $L'$, as specified in the statement would exist, then $\cM-\chi_{S'}$, $\cM-2\cdot \chi_{E'}$, or $\cM-4\cdot \chi_{L'}$ would be 
  an $8$-divisible multiset of cardinality $22$, $23$, or $25$, which is impossible due to Lemma~\ref{lemma_8_div_card_22}, 
  Lemma~\ref{lemma_8_div_card_23}, and Proposition~\ref{prop_Gamma_infty}.
  
  Assume that $P$ is a point with $\cM(P)\ge 4$. Proposition~\ref{prop_div_4_card_13} implies $\cM(H)\in \{21,29\}$ for each hyperplane 
  $H$ containing $P$, so that Lemma~\ref{lemma_lower_bound_space_mult} yields $\cM(P)>5$, i.e., $\lambda_4=\lambda_5=0$ and each point with multiplicity 
  at most $3$ indeed has to be contained in a hyperplane of multiplicity $13$. Double counting the points in the hyperplanes containing $P$ yields 
  that $P$ is contained in $2^{k-5}-2$ hyperplanes of multiplicity $29$ if $\cM(P)=6$ and $2^{k-4}-2$ hyperplanes of multiplicity $29$ if $\cM(P)=7$. 
  Thus, we have $k\ge 6$ if $\lambda_6\ge 1$ and $k\ge 5$ if $\lambda_7\ge 1$.  
  
  First we will show $\lambda_6+\lambda_7\le 1$ and $k\ge 5$. Assume that $L$ is a line spanned by two points with multiplicity at least $6$ and denote 
  the third point of $L$ by $P$. From Proposition~\ref{prop_div_4_card_13} we conclude that each hyperplane $H$ with multiplicity $13$ meets $L$ exactly in $P$, 
  so that $a_{13}\le [k-1]_2-[k-2]_2=2^{k-2}<5\cdot 2^{k-4}+2+a_{29}$ -- contradiction. Thus, we have $\lambda_6+\lambda_7\in\{0,1\}$. From $a_{13}\in \N$ we 
  conclude $k\ge 4$. If $k=4$, then Equation~(\ref{eq_lambda_sum_8_div_card_37}) implies $\lambda_6+\lambda_7\ge 1$, which is possible for $k\ge 5$ only. 
  
  Assume that $P$ is a point of multiplicity at least $6$. Further assume the existence of a line $L$ not containing $P$ but whose three points all 
  have multiplicity at least $2$. 
  By $E$ we denote the plane spanned by $L$ and $P$, so that $\cM(H)\ge 21$ 
  for every hyperplane $H$ containing $E$, see Proposition~\ref{prop_div_4_card_13}. If $\cM(H)=21$, then $\cM|_H-2\cdot\chi_L-4\cdot\chi_P$ is $4$-divisible with cardinality $11$, so that 
  Proposition~\ref{prop_4_div_card_11} implies the existence of a point $Q$ with multiplicity $4$ in $\cM|_H-2\cdot\chi_L-4\cdot\chi_P$. Since $\gamma_1(\cM)<8$ and 
  $\lambda_4+\lambda_5+\lambda_6+\lambda_7\le 1$ this is 
  impossible. Thus, all   $2^{k-3}-1$ hyperplanes containing $E$ have multiplicity $29$. Now let $\tilde{L}\le E$ be a line with $P\le \tilde{L}$, so that we have $\cM(H)\in\{21,29\}$ for every 
  hyperplane containing $\tilde{L}$. Denoting the number of those hyperplanes with multiplicity $29$ by $x$ and double counting points gives
  $$
    \left(2^{k-3}-1\right)\cdot(37-\cM(\tilde{L}))=\left(2^{k-2}-1\right)\cdot (21-\cM(\tilde{L}))+8x,
  $$  
  so that $\cM(\tilde{L})=5+2^{7-k}+2^{6-k}x$. Using $x\ge 2^{k-3}-1$ we conclude $\cM(\tilde{L})\ge 13+2^{6-k}$. However, $\cM(\tilde{L})\le 7+2\cdot 3=13$ -- contradiction.
  
  So, if $P$ is a point with multiplicity at least $6$, then any line $L$ that does not contain $P$ contains a point of multiplicity at most $1$. With this, Proposition~\ref{prop_div_4_card_13} 
  yields that any hyperplane $H$ with multiplicity $13$ contains of a unique point of multiplicity $3$ and $10$ points of multiplicity $1$. Especially, we have $\cM(P')\neq 2$ 
  for every point $P'$ and $k\ge 6$. If $k\ge 7$, then let $H_1$ be a hyperplane with multiplicity $13$ and $K\le H_1$ a $(k-2)$-dimensional subspace with multiplicity $\cM(K)=13$. With this  
  let $H_2$ and $H_3$ be 
  the two other hyperplanes containing $K$. W.l.o.g.\ we assume $\cM(H_2)=21$ and $\cM(H_3)=29$. Clearly $P\not\le H_1$ and since $\cM|_{H_2}-\cM|_{K}-4\cdot\chi_P$ would be $4$-divisible 
  with cardinality $4$, $P$ is also not contained in $H_2$. However, $\cM|_{H_3}-\cM|_{K}-4\cdot\chi_P$ is $4$-divisible with cardinality $12$. By construction, except at most one point, the 
  point multiplicities are contained in $\{0,1,3\}$, which contradicts Proposition~\ref{prop_4_div_card_12}. It remains to consider the case $k=6$. If $P_6$ is a point of multiplicity $6$, then    
  it is contained in $2^{k-5}-2=0$ hyperplanes of multiplicity $29$, so that the line $L$ spanned by $P_6$ and a point of multiplicity $3$ is contained in hyperplanes of multiplicity $21$ 
  only. Lemma~\ref{lemma_lower_bound_space_mult} then yields the contradiction $\cM(L)=7<6+3$. So, let $P_7$ be a point of multiplicity $7$, so that $3\lambda_3 = 4 a_{29}$, which implies 
  $a_{29}\ge 3$ and $\lambda_3\ge 4$. Now let $L$ be a line spanned by two points of multiplicity $3$ such that $P_7$ is not contained in $L$ and $E$ be the plane spanned by 
  $L$ and $P_7$. If $H$ would be a hyperplane with multiplicity $21$, then we can apply Lemma~\ref{lemma_4_div_card_17_special} to $\cM|_H-4\cdot \chi_{P_7}$ -- contradiction.
  So, Lemma~\ref{lemma_lower_bound_space_mult} yields $\cM(E)=23$ and $\cM(S)=25$ for each solid containing $E$. Moreover, we have $a_{29}\ge 7$, so that $a_{29}\ge 9$ and $\lambda_3\ge 12$, 
  which is impossible.   
   
  Thus, we finally conclude $\gamma_1(\cM)\le 3$.
  
  \bigskip    
       
  If $k=5$, then the existence of a hyperplane of multiplicity $13$ implies $\lambda_1\ge 4$, see Proposition~\ref{prop_div_4_card_13}, 
  so that $\lambda_2+3\lambda_3\le 33$. With this, Equation~(\ref{eq_lambda_sum_8_div_card_37}) yields $\lambda_1=4$, $\lambda_2=0$, $\lambda_3=11$, and $a_{29}=0$. 
  However, the four points of multiplicity $1$ span a unique plane, which contradicts $a_{13}\ge 12 > 3$.     
    
\end{proof}

\begin{lemma}
  \label{lemma_8_div_card_37}
  Let $\cM$ be an $8$-divisible multiset of points in $\PG(v-1,2)$ with cardinality $37$. Then, we have $\gamma_1(\cM)\ge 8$. 
\end{lemma}
\begin{proof}
  Assume that $\cM$ is an $8$-divisible multiset of points in $\PG(v-1,2)$ with cardinality $37$, $\gamma_1(\cM)\le 7$, and minimum possible dimension $k$ of its span.  
  Lemma~\ref{lemma_8_div_card_37_aux_1} yields $\gamma_1(\cM)\le 3$ and $k\ge 6$, so that there clearly exists a point $Q$ with multiplicity zero. Using Lemma~\ref{lemma_projection} 
  we conclude that the projection $\cM_Q$ through $Q$ is an $8$-divisible multiset of points with cardinality $37$, $\gamma_1\!\left(\cM_Q\right)\le 2\cdot \gamma_1(\cM)\le 6$, 
  and dimension $k-1$ of its span, which contradicts the minimality of $k$. 
\end{proof}

\begin{lemma}
  Let $\cM$ be an $8$-divisible multiset of points in $\PG(v-1,2)$ with cardinality $20$. Then, we have $\gamma_1(\cM)\ge 4$.
\end{lemma}
\begin{proof}
  The possible hyperplane multiplicities are given by $12$ and $4$. If the latter occurs,
  then there is a point of multiplicity $4$, see Corollary~\ref{cor_simplex_repetition_special}. 
  Otherwise we can apply Lemma~\ref{lemma_repeated_simplex} to obtain a contradiction.
\end{proof}

\begin{lemma}
  \label{lemma_8_div_card_24}
  Let $\cM$ be an $8$-divisible multiset of points in $\PG(v-1,2)$ with cardinality $24$. Then, we have $\gamma_1(\cM)\ge 4$.
\end{lemma}
\begin{proof}
  W.l.o.g.\ we assume $2\le \gamma_1(\cM)\le 3$. The possible hyperplane multiplicities are given by $0$, $8$, and $16$. 
  Let $P$ be a point with multiplicity $\cM(P)=3$. Due to Proposition~\ref{prop_multiples_of_affine_spaces} we have $\cM(H)=16$ 
  for every hyperplane $H$ containing $P$, so that Lemma~\ref{lemma_lower_bound_space_mult} yields $\cM(P)\ge  9$ -- contradiction. 
  Thus, we have $\gamma_1(\cM)=2$ and denote the dimension of the span of $\cM$ by $k$. Using the standard equations we compute
  \begin{eqnarray*}
    a_0&=& a_{16}-2^{k - 1} + 2,\\
    a_8&=& -2a_{16} + 3\cdot 2^{k - 1} - 3,\text{ and}\\
    \lambda_2&=&-108 +(a_{16}+3)\cdot 2^{8-k},
  \end{eqnarray*}
  so that
  $$
    a_{16}\ge 2^{k-1}-2\quad\text{and}\quad \lambda_2\ge 20+2^{8-k}>20,
  $$
  which contradicts $\lambda_2\le 24/2=12$.
\end{proof}

\begin{lemma}
  \label{lemma_8_div_card_26}
  Let $\cM$ be an $8$-divisible multiset of points in $\PG(v-1,2)$ with cardinality $26$. Then, we have $\gamma_1(\cM)\ge 4$.
\end{lemma}
\begin{proof}
  W.l.o.g.\ we assume $2\le \gamma_1(\cM)\le 3$. The possible hyperplane multiplicities are given by $18$ and $10$. 
  If there exists a point $P$ with multiplicity $\cM(P)=3$, then Corollary~\ref{cor_4_div_card_10} implies $\cM(H)=18$ for 
  all hyperplanes $H$ that contain $P$. With this Lemma~\ref{lemma_lower_bound_space_mult} yields $\cM(P)\ge 11$ -- contradiction. 
  Thus, we have $\gamma_1(\cM)=2$ and the standard equations yield $\lambda_2=17+2^{8-k}$  for the 
  dimension $k$ of the span of $\cM$, which clearly contradicts $\#\cM=26$.  
\end{proof}

\begin{lemma}
  \label{lemma_8_div_card_27}
  Let $\cM$ be an $8$-divisible multiset of points in $\PG(v-1,2)$ with cardinality $27$. Then, we have $\gamma_1(\cM)\ge 4$.
\end{lemma}
\begin{proof}
  The possible hyperplane multiplicities are given by $19$ and $11$. If the latter occurs, 
  then there is a point of multiplicity at least $4$, see Proposition~\ref{prop_4_div_card_11}.
  Otherwise we can apply Lemma~\ref{lemma_repeated_simplex} to obtain a contradiction.
\end{proof}

\begin{lemma}
  \label{lemma_8_div_card_35}
  Let $\cM$ be an $8$-divisible multiset of points in $\PG(v-1,2)$ with cardinality $35$. Then, we have $\gamma_1(\cM)\ge 4$.
\end{lemma}
\begin{proof}
  The possible hyperplane multiplicities are given by $11$, $19$, and $27$. However, if there exists a hyperplane $H$ with
  $\cM(H)=11$, then Proposition~\ref{prop_4_div_card_11} yields the existence of a point with multiplicity at least $4$. 
  Otherwise Lemma~\ref{lemma_lower_bound_space_mult} gives $\cM(P)>3$ for every point $P$.
\end{proof}

\begin{lemma}
  \label{lemma_8_div_card_39}
  Let $\cM$ be an $8$-divisible multiset of points in $\PG(v-1,2)$ with cardinality $39$. Then, we have $\gamma_1(\cM)\ge 4$.
\end{lemma}
\begin{proof}
  Assume that $\cM$ is an $8$-divisible multiset of points in $\PG(v-1,2)$ with cardinality $39$, $\gamma_1(\cM)\le 3$, and minimum possible dimension $k$ of its span. 
  If $Q$ would be a point with multiplicity $2$, then the projection $\cM_Q$ through $Q$ would be $8$-divisible with cardinality $37$ and $\gamma_1\!\left(\cM_Q\right)\le 6$, 
  so that Lemma~\ref{lemma_8_div_card_37} yields a contradiction. Thus, we have $\lambda_2=0$ and $\gamma_1(\cM)=3$, see Proposition~\ref{prop_Gamma_1}. 
  Since $\gamma_1(\cM)\le 3$ and $\#\cM=39$, we clearly have $k\ge 4$. If $k=4$, then the multiset of points $\cM'$ defined by $\cM'(P)=3-\cM(P)$ is $8$-divisible with 
  cardinality $6$ -- contradiction. Thus, we have $k\ge 5$. Since each hyperplane $H$ with multiplicity $7$ is given by $\cM|_H=\chi_E$ for some plane $E$, see 
  Corollary~\ref{cor_simplex_repetition_special}, each hyperplanes $H'$ that contains a point of multiplicity $3$ has multiplicity $\cM(H')\ge 15$. 
  Lemma~\ref{lemma_lower_bound_space_mult} yields that each point of multiplicity $3$ is contained in a hyperplane $H$ of multiplicity $15$. Using
  $\gamma_1(\cM)\le 3$ and $\lambda_2=0$, Corollary~\ref{cor_4_div_card_15} yields $\cM|_H=3\cdot\chi_E-2\cdot\chi_L$ to a plane $E$ and a line $L\le E$. So, $\lambda_3\ge 1$ implies 
  $\lambda_3\ge 4$.
  
  Let $H$ be hyperplane with multiplicity $15$ and $E$ a plane, $L\le E$ a line such that $\cM|_H= 3\cdot\chi_E-2\cdot\chi_L$. Choose a subspace $K\le H$ of codimension $2$ intersecting 
  $E$ in a line of multiplicity $7$, i.e., containing two points of multiplicity $3$ and one point of multiplicity $1$. We denote the other two hyperplanes containing $K$ by 
  $H'$ and $H''$. W.l.o.g.\ we assume $\cM(H')=15$ and $\cM(H'')=23$. So, let $E'\le H'$ a plane and $L'\le E'$ a line such that $\cM|_{H'}= 3\cdot\chi_{E'}-2\cdot\chi_{L'}$. We 
  conclude $\lambda_3\ge 6$, $\lambda_1\ge 5$ and denote the solid spanned by $E$ and $E'$ by $S$. By construction we have $\cM(S)\ge 23$. However, we have $\cM(H)\neq 23$ for any 
  hyperplane $H\ge S$ since $\cM|_H-3\cdot\chi_E+2\cdot\chi_L$ would be $4$-divisible with cardinality $8$ containing two points of multiplicity $3$, which contradicts 
  Proposition~\ref{prop_multiples_of_affine_spaces}. Thus, we have $\cM(H)=31$ for every hyperplane $H$ that contains $S$, so that Lemma~\ref{lemma_lower_bound_space_mult} 
  yields $\cM(S)=23+2^{8-k}$. If $k=5$, then $\cM(S)=31$. However, since no three points of multiplicity $3$ form a line, $S$ can contain at most $8$ points of multiplicity $3$ and 
  seven points of multiplicity $1$, which yields $\cM\ge \chi_S$ and $\cM- \chi_S$ is $8$-divisible of cardinality $31$, so that Lemma~\ref{lemma_8_div_card_24} yields a contradiction.
  
  Let $H$ be an arbitrary but fixed hyperplane of multiplicity $15$ and $E$ be the corresponding plane that contains the non-zero points of $\cM|_H$. Now consider subspace $K\le H$ of codimension 
  $2$. In $2^{k-4}-1$ cases $E\le K$ and the other two hyperplanes containing $K$ have multiplicities $23$ and $31$. In $6\cdot 2^{k-4}$ cases $K$ intersects $E$ in a line of multiplicity $7$ and 
  the other two hyperplanes containing $K$ have multiplicities $15$ and $23$. In $2^{k-4}$ cases $K$ intersects $E$ and a line of multiplicity $3$ and there are two cases for the other two
  hyperplanes $H'$, $H''$ containing $K$. Either $\left\{\cM(H'),\cM(H'')\right\}=\{7,23\}$ or $\cM(H')=\cM(H'')=15$. Denote the number of occurrence of the first case by $x$, so that $x\in\N$
  with $x\le 2^{k-4}$. With this, we compute
  \begin{eqnarray*}
    a_7&=&x,\\
    a_{15} &=& 1+8\cdot 2^{k-4}-2x,\\ 
    a_{23} &=& 7\cdot 2^{k-4}-1+x,\text{ and}\\
    a_{31} &=& 2^{k-4}-1,
  \end{eqnarray*}   
  so that $2^{9-k}+28+2^{8-k}x=3\lambda_3$, which implies $\lambda_3\ge 10$ and $\lambda_1\le 9$. However, in $E$ there are three points of multiplicity $1$ and there exists a hyperplane $H\ge E$ 
  with $\cM(H)=23$, so that Proposition~\ref{prop_multiples_of_affine_spaces}, $\gamma_1(\cM)\le 3$, and $\lambda_2=0$ imply $\lambda_1\ge 3+8=11>9$ -- contradiction.
\end{proof}

\begin{lemma}
  \label{lemma_8_div_card_41_aux_1}
  Let $\cM$ be a spanning $8$-divisible multiset of points in $\PG(v-1,2)$ with cardinality $41$ and $1\le \gamma_1(\cM)\le 3$. Then, we have
  \begin{eqnarray*}
    a_{17} &=& 9\cdot 2^{k-4}+2+a_{33},\\ 
    a_{25} &=& 7\cdot 2^{k-4}-3-2a_{33},\\ 
    \lambda_2+3\lambda_3  &=& 11 +2^{8-k}\cdot \left(3+a_{33}\right),
  \end{eqnarray*}
  and $k\ge 5$. Moreover, for each line $L$ there exists a hyperplane $H\ge L$ with $\cM(H)=17$ and a point $P\le L$ with $\cM(P)\le 1$.
\end{lemma}
\begin{proof}
  Since no $4$-divisible multiset of points of cardinality $9$ exists, the possible multiplicities of the hyperplanes are given by $17$, $25$, and $33$. Using the standard equations we compute
  the stated equations. Since $\gamma_1(\cM)\le 3$, we clearly have $k\ge 4$. If $k=4$, then the multiset of points $\cM'$ defined by $\cM'(P)=3-\cM(P)$ is $8$-divisible with cardinality $4$, 
  which is impossible. 
  
  Let $L$ be an arbitrary line. If $L$ is not contained in a hyperplane of multiplicity $17$, then Lemma~\ref{lemma_lower_bound_space_mult} yields $\cM(L)>25-16=9$. However, $\gamma_1(\cM)\le 3$
  implies $\cM(L)\le 9$ -- contradiction. So, let $H\ge L$ be a hyperplane with $\cM(H)=17$. If $\cM(P)\ge 2$ for all points $P\le L$, then $\cM|_H-2\cdot \chi_L$ is $4$-divisible with 
  cardinality $11$, so that Proposition~\ref{prop_4_div_card_11} contradicts $\gamma_1(\cM)\le 3$. 
\end{proof}

\begin{lemma}
  \label{lemma_8_div_card_41_aux_2}
  Let $\cM$ be a spanning $8$-divisible multiset of points in $\PG(v-1,2)$ with cardinality $41$ and $1\le \gamma_1(\cM)\le 3$. Then, there does not exist a 
  hyperplane $H$ such that $\cM|_H$ is given as specified in Lemma~\ref{lemma_4_div_card_17_special_2}.
\end{lemma}
\begin{proof}
  Assume that $\cM$ is a spanning $8$-divisible multiset of points in $\PG(k-1,2)$ with cardinality $41$ and $1\le \gamma_1(\cM)\le 3$, such that $\cM|_{H_1}$ 
  is as specified in Lemma~\ref{lemma_4_div_card_17_special_2} for some hyperplane $H_1$. Since $\cM|_{H_1}$ spans a $5$-dimensional subspace we have $k\ge 6$ for the 
  dimension of the point set spanned by $\cM$. Let $E:=\left\langle e_{7},\dots, e_k\right\rangle$, where we also allow $E$ to be an empty space for $k=6$. W.l.o.g.\ we 
  also assume coordinates as in Lemma~\ref{lemma_4_div_card_17_special_2}, so that especially $H_1=\left\langle e_1,\dots,e_5,E\right\rangle$. Consider the subspace 
  $S:=\left\langle e_1,e_2,e_4,e_5,E\right\rangle\le H_1$ with $\cM(S)=9$, so that we have $\cM(H_2)=17$ and $\cM(H_3)=25$ for the two other hyperplanes containing $S$. 
  Note that the line $L:=\left\langle e_1,e_2\right\rangle$ is contained in $H_2\ge S$ and consists of a point of multiplicity $i$ for all $1\le i\le 3$, so that we can 
  apply Lemma~\ref{lemma_4_div_card_17_special_2} to $\cM|_{H_2}$. W.l.o.g.\ we assume that the second point of multiplicity $3$ in $H_2$ (that is not contained in $S$) has 
  coordinates $e_6$. With this, the point set $\cM|_{H_1\cup H_2}$ is given by the columns of 
  $$
  \begin{pmatrix}
  1 1 1 1 1 1 1 1 0 0 0 0 0 0 0 0 0 0 0 0 1 0 1 1 1\\
  0 0 0 1 1 1 1 0 1 1 0 0 0 0 0 1 1 0 0 0 1 1 1 1 0\\
  0 0 0 0 1 1 1 1 0 0 1 0 1 0 1 0 1 0 0 0 0 0 0 0 0\\
  0 0 0 0 0 0 0 0 0 0 0 1 1 0 0 1 1 0 0 0 0 0 1 0 1\\
  0 0 0 0 0 0 0 0 0 0 0 0 0 1 1 1 1 0 0 0 0 0 0 1 1\\
  0 0 0 0 0 0 0 0 0 0 0 0 0 0 0 0 0 1 1 1 1 1 1 1 1
  \end{pmatrix}\!\!,
  $$
  where the entries in the rows $7$ to $k$ are all $0$ and not displayed. This multiset of points consists of $14$ points of multiplicity $1$, a unique point of multiplicity 
  $2$, and three points of multiplicity $3$.
  
  There exist at least
  $$
    \Omega:=2^k-1 -\left(\lambda_1+\lambda_2+\lambda_3\right) -{{\lambda_2+\lambda_3}\choose 2}-\lambda_3\cdot \lambda_1
  $$  
  points $Q$ of multiplicity $0$ such that every line that contains $Q$ has multiplicity at most $3$. Assume $k\ge 8$. If $\lambda_3\ge 3+4=7$, then 
  we have $\left(\lambda_1,\lambda_2,\lambda_3\right)\in\left\{(18,1,7),(16,2,7),(14,3,7),(15,1,8)\right\}$, where $\Omega\ge 38>0$. If $\lambda_3 <7$, then 
  $\lambda_2+\lambda_3\le 12$ implies $\Omega\ge 255 -30- 66- 126=43>0$. So, such a point $Q$ exists and the projection $\cM_Q$ of $\cM$ through $Q$ is 
  $8$-divisible with cardinality $41$ and $\gamma_1(\cM_P)\le 3$, see Lemma~\ref{lemma_projection}. W.l.o.g.\ we can assume that $k$ is minimal, so that it suffices 
  to consider the cases $k\in\{6,7\}$ in the following.   
  
  If $k=6$, then Lemma~\ref{lemma_8_div_card_41_aux_1} yields $\lambda_2+3\lambda_3  = 23 +4a_{33}$. Since the points in $H_1\cup H_2$ contribute $1\cdot 1+3\cdot 3=10$ to 
  $\lambda_2+3\lambda_3$ and the points outside of $H_1\cup H_2$ can contribute at most $16$, we have $a_{33}=0$ and $\lambda_2\equiv 2\pmod 3$, which implies 
  $\lambda_2=2$, $\lambda_3=7$, and $\lambda_1=16$. Let $P_2'$ denote the second point of multiplicity $2$, which lies outside of $S$. For the plane 
  $E':=\left\langle P_2',L\right\rangle$ we know that all hyperplanes $H$ that contain $E'$ have multiplicity $25$, so that Lemma~\ref{lemma_lower_bound_space_mult} 
  implies $\cM(E')=13$ and $\cM(S')=17$ for every solid $S'\ge E'$. From Lemma~\ref{lemma_8_div_card_41_aux_1} we know that the points with multiplicity at least $2$ are contained in an affine plane, so that 
  $E'$ consists of two points of multiplicity $2$, two points of multiplicity $3$, and three points of multiplicity $1$ forming a line. Starting from 
  $L'':=\left\langle e_2,e_1+e_2+e_3\right\rangle$ we can construct $E'':=\left\langle P_2',L''\right\rangle$ and deduce the same structure information for $E''$. 
  Since both $E'$ and $E''$ contain the two points of multiplicity $2$, their span $S':=\left\langle E',E''\right\rangle$ is a solid with multiplicity at 
  least $\cM(E')+2\cdot 3\ge 19>17$ -- contradiction.   
  
  If $k=7$, then Lemma~\ref{lemma_8_div_card_41_aux_1} yields $\lambda_2+3\lambda_3  = 17 +2a_{33}$. We choose $K:=\left\langle e_1,\dots,e_r\right\rangle$ and 
  $H_1:=\left\langle K, e_7\right\rangle$, so that $\cM(K)=\cM\!\left(H_1\right)=17$ and $\left\{\cM(H_2),\cM(H_3)\right\}=\{25,33\}$ for the other two 
  hyperplanes that contain $K$. Let $H_2:=\left\langle K,e_6\right\rangle$, so that $e_6\in H_2\backslash K$ is a point of multiplicity $3$ and   
  Proposition~\ref{prop_multiples_of_affine_spaces} implies $\cM\!\left(H_2\right)=33$. Thus, $\cM':=\cM|_{H_2}-\cM|_K$ is $4$-divisible with cardinality $16$ 
  and contains at least one point of multiplicity $3$ as well as five points of multiplicity $1$. Since there can be at most $\left\lfloor 8/3\right\rfloor=2$ more
  points of multiplicity $3$, Lemma~\ref{lemma_4_div_card_16_special} implies that $\cM'$ contains at most two points of multiplicity $3$ in total. 
  Since $\cM|_{H_3}-\cM|_{K}$ is $4$-divisible with cardinality $8$, Proposition~\ref{prop_multiples_of_affine_spaces} implies $3\le \lambda_3 \le 4$. 
  Combining this with $\lambda_2+3\lambda_3  = 17 +2a_{33}$ and $a_{33}\ge 1$ we conclude $\lambda_2\ge 7$. Using Proposition~\ref{prop_multiples_of_affine_spaces} 
  again we conclude $1\le \lambda_3\left(\cM'\right)$ and $\lambda_2\left(\cM'\right)\ge 2$, so that we can apply Lemma~\ref{lemma_4_div_card_16_special2}. 
  If $\lambda_3=3$, then $\lambda_2\le 1+3+4=8$, which contradicts $\lambda_2+3\lambda_3\ge 19$. Thus, we have $\lambda_3=4$ and the upper bound $\lambda_2\le 1+2+4$ 
  implies $a_{33}=1$, $\lambda_2=7$, and $\lambda_1=15$ using $\lambda_2+3\lambda_3  = 17 +2a_{33}$ and $\lambda_1+2\lambda_2+3\lambda_3=41$. Moreover, the six points 
  of multiplicity $1$ and the two points of multiplicity $3$ outside of $H_1$ form an affine solid. However, we have
  $$
    \left\langle e_6,e_2+e_6,e_1+e_2+e_6,e_1+e_2+e_4+e_6,e_1+e_2+e_5+e_6,e_1+e_4+e_5+e_6\right\rangle=\left\langle e_1,e_2,e_4,e_5,e_6\right\rangle,
  $$  
  which is a contradiction.
\end{proof}

\begin{lemma}
  \label{lemma_8_div_card_41}
  Let $\cM$ be an $8$-divisible multiset of points in $\PG(v-1,2)$ with cardinality $41$. Then, we have $\gamma_1(\cM)\ge 4$.
\end{lemma}
\begin{proof}
  Assume that $\cM$ is a spanning $8$-divisible multiset of points in $\PG(k-1,2)$ with cardinality $41$, $\gamma_1(\cM)\le 3$, and minimum possible dimension $k$ 
  of its span. Proposition~\ref{prop_Gamma_1} yields $\gamma_1(\cM)\ge 2$.
  
  Assume that $P_2$ is a point with multiplicity $2$. If $L\ge P_2$ is a line that also contains a point of multiplicity $3$ and a point of multiplicity $1$, then
  Lemma~\ref{lemma_8_div_card_41_aux_1} implies the existence of a hyperplane $H\ge L$ with $\cM(H)=17$ and Lemma~\ref{lemma_4_div_card_17_special_2} 
  yields a description of $\cM|_H$. However, Lemma~\ref{lemma_8_div_card_41_aux_2} gives a contradiction. Using Lemma~\ref{lemma_8_div_card_41_aux_1} again we conclude 
  $\cM(L')\le 5$ for each line $L'\ge P_2$. So, for the projection $\cM_{P_2}$ of $\cM$ through $P_2$ we have $\gamma_1\!\left(\cM_{P_2}\right)\le 3$. However, 
  $\cM_{P_2}$ is $8$-divisible with cardinality $39$, which contradicts Lemma~\ref{lemma_8_div_card_39}. Thus, we have $\lambda_2=0$ and $\gamma_1(\cM)=3$.

  Lemma~\ref{lemma_8_div_card_41_aux_1} gives $3\lambda_3  = 11 +2^{8-k}\cdot \left(3+a_{33}\right)$, so that $\lambda_3\ge 4$. Let $L$ be a line spanned by two points 
  of multiplicity $3$. Lemma~\ref{lemma_8_div_card_41_aux_1} yields the existence of a hyperplane $H\ge L$ with $\cM(H)=17$, so that Lemma~\ref{lemma_4_div_card_17_special_3} 
  gives $\cM(L)=6$, $\lambda_3\!\left(\cM|_H\right)=2$, $\lambda_1\!\left(\cM|_H\right)=11$, and $k\ge 7$ (for the dimension of $\cM$). Since $\lambda_1\ge 11$, we have $\lambda_3\le 10$. So, $4\le \lambda_3\le 10$
  implies $\lambda_1+\lambda_3 \le 33$, ${\lambda_3\choose 2}\le 45$, and $\lambda_1\cdot\lambda_3\le 140$. Since $33+45+140< 2^8-1$ for $k\ge 8$, there exists a point 
  $Q$ of multiplicity zero such that every line $L'\ge Q$ has multiplicity at most $3$. With this, the projection $\cM_Q$ of $\cM$ through $Q$ is $8$-divisible with 
  cardinality $8$ and $\gamma_1(\cM_Q)=3$, see Lemma~\ref{lemma_projection}. Due to the assumed minimality of $k$ we have $k=7$.

  Using $k=7$, the equation $3\lambda_3  = 11 +2^{8-k}\cdot \left(3+a_{33}\right)=17+2a_{33}$ implies $\lambda_3\in \{7,9\}$ and $a_{33}\in\{2,5\}$. Due to Lemma~\ref{lemma_4_div_card_17_aux} 
  the hyperplane $H$ with multiplicity $17$ contains a $5$-dimensional subspace $K$ with multiplicity $\cM(K)=5$. Since also the two other hyperplanes that contain $K$ then
  also have multiplicity $17$, Lemma~\ref{lemma_4_div_card_17_special_3} yields $\lambda_3\le 3\cdot 2=6$ -- contradiction.
\end{proof}

\section{Conclusion}
\label{sec_conclusion}

We have determined the minimum possible column multiplicities for $\Delta$-divisible binary linear codes for each given length $n$ and all $\Delta\in\{2,4,8\}$. This refines
a comprehensive characterization result on the possible length of $q^r$-divisible linear codes over $\mathbb{F}_q$ from \cite{kiermaier2020lengths}. The motivation 
for this refinement is that in some applications upper bounds on the allowed maximum column multiplicity are given. We mainly use geometric methods to obtain computer-free 
proofs. While the stated result can also be obtained by an exhaustive computer enumeration, the question arises whether the theoretical tools can be strengthened and approaches 
be simplified in order to obtain results for wider ranges of parameters. As outlined in Section~\ref{sec_computational_results} in the appendix, currently we cannot go much 
further even using extensive computer enumerations. Interestingly enough $\Gamma_2(2^r,n)$ is always a power of $2$, if finite at all, for $r\in\{1,2,3\}$. Our, rather sparse, numerical
data might suggest the conjecture that $\Gamma_q(q^r,n)$ always has to be a power of the characteristic $p$ of the underlying field $\mathbb{F}_q$. However, such a strong statement 
is wrong in general. To this end, note that applying the construction of \cite[Theorem 10]{landjevr2019} to the smallest non-trivial blocking set in $\PG(2,p)$, i.e.,the projective triangle, 
yields a $p$-divisible multiset $\cM$ of points with cardinality $\#\cM=\tfrac{p^2+1}{2} +2p$ and maximum point multiplicity $\gamma_1(\cM)=\tfrac{p+3}{2}$ for each odd prime $p$. So, 
for $p\ge 5$ we have $\Gamma_p(p,\#\cM)<p$ while \cite[Theorem 11]{honold2018partial} yields $\Gamma_p(p,\#\cM)>1$. 


\newcommand{\etalchar}[1]{$^{#1}$}

\appendix

\section{Combinatorial data of the possible $4$-divisible multisets of points with cardinality $17$}
\label{sec_data}

In our discussion of $8$-divisible multisets of points of cardinality $41$, $4$-divisible multisets 
of points of cardinality $17$ played an important role. To this end we have stated several auxiliary
results for the latter.

\begin{table}[htp!]
  \begin{center}
  \begin{tabular}{rrrr|rrr|cc}
  $n$ & $k$ & $\Delta$ & $\gamma_1$ & $\lambda_1$ & $\lambda_2$ & $\lambda_3$ & spectrum & \\ \hline
17 & 3 & 4 & 7 &  4 &  0 &  2 & $(a_{5},a_{9},a_{13})=(4,2,1)$ & \\
17 & 3 & 4 & 5 &  3 &  0 &  3 & $(a_{5},a_{9},a_{13})=(3,4,0)$ & \\ \hline
17 & 4 & 4 & 7 &  6 &  2 &  0 & $(a_{5},a_{9},a_{13})=(8,3,4)$ & \\
17 & 4 & 4 & 6 &  6 &  1 &  1 & $(a_{5},a_{9},a_{13})=(7,5,3)$ & \\
17 & 4 & 4 & 5 &  5 &  2 &  1 & $(a_{5},a_{9},a_{13})=(6,7,2)$ & \\
17 & 4 & 4 & 4 &  6 &  2 &  1 & $(a_{5},a_{9},a_{13})=(5,9,1)$ & \\
17 & 4 & 4 & 4 &  4 &  0 &  3 & $(a_{5},a_{9},a_{13})=(6,7,2)$ & \\
17 & 4 & 4 & 3 &  4 &  2 &  3 & $(a_{5},a_{9},a_{13})=(5,9,1)$ & \\
17 & 4 & 4 & 3 &  6 &  4 &  1 & $(a_{5},a_{9},a_{13})=(4,11,0)$ & \\ \hline 
17 & 5 & 4 & 7 & 10 &  0 &  0 & $(a_{5},a_{9},a_{13})=(16,5,10)$ & \\
17 & 5 & 4 & 6 &  7 &  2 &  0 & $(a_{5},a_{9},a_{13})=(14,9,8)$ & \\
17 & 5 & 4 & 5 &  9 &  0 &  1 & $(a_{5},a_{9},a_{13})=(12,13,6)$ & \\
17 & 5 & 4 & 5 &  6 &  3 &  0 & $(a_{5},a_{9},a_{13})=(12,13,6)$ & \\
17 & 5 & 4 & 4 &  7 &  3 &  0 & $(a_{5},a_{9},a_{13})=(10,17,4)$ & \\
17 & 5 & 4 & 4 & 10 &  0 &  1 & $(a_{5},a_{9},a_{13})=(10,17,4)$ & \\
17 & 5 & 4 & 4 &  6 &  2 &  1 & $(a_{5},a_{9},a_{13})=(11,15,5)$ & \\
17 & 5 & 4 & 3 &  5 &  3 &  2 & $(a_{5},a_{9},a_{13})=(10,17,4)$ & \\
17 & 5 & 4 & 3 &  9 &  1 &  2 & $(a_{5},a_{9},a_{13})=(9,19,3)$ & \\
17 & 5 & 4 & 3 & 10 &  2 &  1 & $(a_{5},a_{9},a_{13})=(8,21,2)$ & \\
17 & 5 & 4 & 3 &  6 &  4 &  1 & $(a_{5},a_{9},a_{13})=(9,19,3)$ & \\
17 & 5 & 4 & 2 &  7 &  5 &  0 & $(a_{5},a_{9},a_{13})=(8,21,2)$ & \\
17 & 5 & 4 & 2 & 11 &  3 &  0 & $(a_{5},a_{9},a_{13})=(7,23,1)$ & \\
17 & 5 & 4 & 2 & 15 &  1 &  0 & $(a_{5},a_{9},a_{13})=(6,25,0)$ & \\ \hline
17 & 6 & 4 & 4 & 10 &  0 &  1 & $(a_{5},a_{9},a_{13})=(21,31,11)$ & \\
17 & 6 & 4 & 4 &  7 &  3 &  0 & $(a_{5},a_{9},a_{13})=(21,31,11)$ & \\
17 & 6 & 4 & 3 & 11 &  0 &  2 & $(a_{5},a_{9},a_{13})=(18,37,8)$ & \\
17 & 6 & 4 & 3 & 10 &  2 &  1 & $(a_{5},a_{9},a_{13})=(17,39,7)$ & \\
17 & 6 & 4 & 3 &  6 &  4 &  1 & $(a_{5},a_{9},a_{13})=(19,35,9)$ & \\
17 & 6 & 4 & 3 & 12 &  1 &  1 & $(a_{5},a_{9},a_{13})=(16,41,6)$ & \\
17 & 6 & 4 & 2 &  7 &  5 &  0 & $(a_{5},a_{9},a_{13})=(17,39,7)$ & \\
17 & 6 & 4 & 2 & 11 &  3 &  0 & $(a_{5},a_{9},a_{13})=(15,43,5)$ & \\
17 & 6 & 4 & 2 & 13 &  2 &  0 & $(a_{5},a_{9},a_{13})=(14,45,4)$ & \\
17 & 6 & 4 & 2 & 15 &  1 &  0 & $(a_{5},a_{9},a_{13})=(13,47,3)$ & \\
17 & 6 & 4 & 1 & 17 &  0 &  0 & $(a_{5},a_{9},a_{13})=(12,49,2)$ & \\ \hline
17 & 7 & 4 & 2 & 11 &  3 &  0 & $(a_{5},a_{9},a_{13})=(31,83,13)$ & \\
17 & 7 & 4 & 2 &  7 &  5 &  0 & $(a_{5},a_{9},a_{13})=(35,75,17)$ & \\
17 & 7 & 4 & 2 & 13 &  2 &  0 & $(a_{5},a_{9},a_{13})=(29,87,11)$ & \\
17 & 7 & 4 & 2 & 15 &  1 &  0 & $(a_{5},a_{9},a_{13})=(27,91,9)$ & \\
17 & 7 & 4 & 1 & 17 &  0 &  0 & $(a_{5},a_{9},a_{13})=(25,95,7)$ & \\ \hline
17 & 8 & 4 & 1 & 17 &  0 &  0 & $(a_{5},a_{9},a_{13})=(51,187,17)$ & \\
  \end{tabular}
  \caption{Combinatorial data of $4$-divisible multisets of points of cardinality $17$.}
  \label{table_combinatorial_data_4_div_card_17}
  \end{center}
\end{table}  

In order to demonstrate the combinatorial richness we have listed key parameters of these objects in Table~\ref{table_combinatorial_data_4_div_card_17}. 
This data has been obtained by an exhaustive computer enumeration using the software package \texttt{LinCode} \cite{bouyukliev2021computer}. The three 
cases of maximum point multiplicity $1$ have also been computationally classified in \cite{ubt_eref40887}.

\section{Computational results}
\label{sec_computational_results}

One alternative way to prove Theorem~\ref{main_thm} is to use the fact that for every field size $q$ and every divisibility constant $\Delta\in\mathbb{N}$ 
there exists an integer $N(q,\Delta)$ such that for all $n\ge N(q,\Delta)$ there exists a projective $\Delta$-divisible linear code over $\mathbb{F}_q$ 
with length $n$. So, for each pair $q$, $\Delta$ a complete determination of the function $\Gamma_q(\Delta,\cdot)$ amounts to a finite computation. So, we
may simply enumerated all $\Delta$-divisible codes over $\mathbb{F}_q$ with length strictly smaller than $ N(q,\Delta)$ and determine the corresponding column 
multiplicities. To this end we have used the software package \texttt{LinCode} \cite{bouyukliev2021computer} and list the corresponding enumeration results 
of semi-linearly non-equivalent linear codes per length and dimension in the subsequent tables. Here a blank entry means that no such code exists.

\begin{table}[htp]
  \begin{center}
    \begin{tabular}{|r|rrrrrrrrr|}
    \hline
    n / k & 1 & 2 & 3 & 4 & 5 & 6 & 7 & 8 & 9\\ 
    \hline
      2 & 1 &   &    &    &    &    &    &   &   \\ 
      3 &   & 1 &    &    &    &    &    &   &   \\
      4 & 1 & 1 &  1 &    &    &    &    &   &   \\
      5 &   & 1 &  1 &  1 &    &    &    &   &   \\
      6 & 1 & 2 &  3 &  2 &  1 &    &    &   &   \\
      7 &   & 2 &  4 &  4 &  2 &  1 &    &   &   \\
      8 & 1 & 3 &  8 & 10 &  7 &  3 &  1 &   &   \\
      9 &   & 3 &  9 & 18 & 16 &  9 &  3 & 1 &   \\ 
     10 & 1 & 4 & 17 & 37 & 46 & 30 & 13 & 4 & 1 \\    
    \hline      
    \end{tabular}
    \caption{Number of even codes per dimension $k$ and effective length $n$.}
  \end{center}
\end{table}

\begin{table}[htp]
  \begin{center}
    \begin{tabular}{|r|rrrrrrrrr|}
    \hline
    n / k & 1 & 2 & 3 & 4 & 5 & 6 & 7 & 8 & 9\\ 
    \hline
     4 & 1 &   &    &    &     &     &     &    &    \\
     6 &   & 1 &    &    &     &     &     &    &    \\
     7 &   &   &  1 &    &     &     &     &    &    \\
     8 & 1 & 1 &  1 &  1 &     &     &     &    &    \\
    10 &   & 1 &  1 &  1 &     &     &     &    &    \\
    11 &   &   &  1 &  1 &     &     &     &    &    \\  
    12 & 1 & 2 &  3 &  4 &   2 &     &     &    &    \\
    13 &   &   &  1 &  1 &   2 &     &     &    &    \\
    14 &   & 2 &  4 &  6 &   5 &   4 &     &    &    \\  
    15 &   &   &  3 &  6 &   6 &   4 &   2 &    &    \\
    16 & 1 & 3 &  8 & 18 &  21 &  15 &   7 &  2 &    \\ 
    17 &   &   &  2 &  7 &  14 &  11 &   5 &  1 &    \\ 
    18 &   & 3 &  9 & 27 &  44 &  45 &  21 &  6 &    \\ 
    19 &   &   &  6 & 22 &  52 &  62 &  40 & 10 &    \\ 
    20 & 1 & 4 & 17 & 64 & 149 & 212 & 156 & 65 & 10 \\ 
    \hline      
    \end{tabular}
    \caption{Number of doubly-even codes per dimension $k$ and effective length $n$.}
  \end{center}
\end{table}

\begin{table}[htp]
  \begin{center}
    \begin{tabular}{|r|rrrrrrrrrrr|}
    \hline
    n / k & 1 & 2 & 3 & 4 & 5 & 6 & 7 & 8 & 9 & 10 & 11 \\ 
    \hline
     8 & 1 &   &    &    &     &     &     &     &    &    &   \\
    12 &   & 1 &    &    &     &     &     &     &    &    &   \\
    14 &   &   &  1 &    &     &     &     &     &    &    &   \\
    15 &   &   &    &  1 &     &     &     &     &    &    &   \\
    16 & 1 & 1 &  1 &  1 &   1 &     &     &     &    &    &   \\
    20 &   & 1 &  1 &  1 &     &     &     &     &    &    &   \\
    22 &   &   &  1 &  1 &     &     &     &     &    &    &   \\
    23 &   &   &    &  1 &   1 &     &     &     &    &    &   \\
    24 & 1 & 2 &  3 &  4 &   4 &   1 &     &     &    &    &   \\
    26 &   &   &  1 &  1 &   2 &     &     &     &    &    &   \\
    27 &   &   &    &  1 &   1 &   1 &     &     &    &    &   \\
    28 &   & 2 &  4 &  6 &   7 &   6 &   1 &     &    &    &   \\
    29 &   &   &    &  1 &   1 &   2 &   1 &     &    &    &   \\
    30 &   &   &  3 &  6 &   8 &   7 &   6 &   2 &    &    &   \\
    31 &   &   &    &  4 &   8 &   8 &   6 &   4 &  1 &    &   \\
    32 & 1 & 3 &  8 & 18 &  32 &  34 &  24 &  13 &  5 &  1 &   \\
    34 &   &   &  2 &  7 &  14 &  11 &   5 &   1 &    &    &   \\
    35 &   &   &    &  3 &   7 &   7 &   3 &   1 &    &    &   \\
    36 &   & 3 &  9 & 27 &  54 &  65 &  36 &  11 &  1 &    &   \\
    37 &   &   &    &  2 &   5 &   8 &   5 &   1 &    &    &   \\
    38 &   &   &  6 & 22 &  57 &  79 &  61 &  21 &  2 &    &   \\
    39 &   &   &    & 10 &  36 &  57 &  49 &  30 & 10 &  1 &   \\
    40 & 1 & 4 & 17 & 64 & 194 & 347 & 323 & 187 & 59 & 11 & 1 \\
    41 &   &   &    &  2 &  12 &  29 &  26 &  12 &  3 &    &   \\
    \hline      
    \end{tabular}
    \caption{Number of triply-even codes per dimension $k$ and effective length $n$.}
  \end{center}
\end{table}

For $16$-divisible binary linear codes we have only partial results. We remark that the smallest 
attained dimension for a given length can be explained by a statement similar to Lemma~\ref{lemma_divisibility_hyperplane}. 
Lengths that do not occur at all are explained by Theorem~\ref{thm_lengths_of_divisible_codes}. The complete classification 
of the possible lengths of projective $16$-divisible binary linear codes is still an open problem, see e.g.\ 
\cite{honold2018partial,honold2019lengths}. The same is true for the projective $q^2$-divisible linear codes over $\mathbb{F}_q$ 
when $q\ge 3$ and the projective $q$-divisible linear codes over $\mathbb{F}_q$ when $q\ge 5$.

\begin{table}[htp]
  \begin{center}
    \begin{tabular}{|r|rrrrrrrrrrrr|}
    \hline
    n / k & 1 & 2 & 3 & 4 & 5 & 6 & 7 & 8 & 9 & 10 & 11 & 12 \\ 
    \hline
    16 & 1 &   &   &    &    &    &    &    &    &    &   &   \\
    24 &   & 1 &   &    &    &    &    &    &    &    &   &   \\
    28 &   &   & 1 &    &    &    &    &    &    &    &   &   \\
    30 &   &   &   &  1 &    &    &    &    &    &    &   &   \\
    31 &   &   &   &    &  1 &    &    &    &    &    &   &   \\
    32 & 1 & 1 & 1 &  1 &  1 &  1 &    &    &    &    &   &   \\
    40 &   & 1 & 1 &  1 &    &    &    &    &    &    &   &   \\
    44 &   &   & 1 &  1 &    &    &    &    &    &    &   &   \\
    46 &   &   &   &  1 &  1 &    &    &    &    &    &   &   \\
    47 &   &   &   &    &  1 &  1 &    &    &    &    &   &   \\ 
    48 & 1 & 2 & 3 &  4 &  4 &  3 &  1 &    &    &    &   &   \\
    52 &   &   & 1 &  1 &  2 &    &    &    &    &    &   &   \\
    54 &   &   &   &  1 &  1 &  1 &    &    &    &    &   &   \\
    55 &   &   &   &    &  1 &  1 &  1 &    &    &    &   &   \\ 
    56 &   & 2 & 4 &  6 &  7 &  8 &  3 &  1 &    &    &   &   \\ 
    58 &   &   &   &  1 &  1 &  2 &  1 &    &    &    &   &   \\
    59 &   &   &   &    &  1 &  1 &  1 &  1 &    &    &   &   \\
    60 &   &   & 3 &  6 &  8 &  9 &  8 &  4 &  1 &    &   &   \\ 
    61 &   &   &   &    &  1 &  1 &  2 &  1 &  1 &    &   &   \\
    62 &   &   &   &  4 &  8 & 10 &  9 &  8 &  4 &  2 &   &   \\
    63 &   &   &   &    &  5 & 10 & 10 &  8 &  6 &  3 & 1 &   \\   
    64 & 1 & 3 & 8 & 18 & 32 & 48 & 48 & 35 & 21 & 11 & 4 & 1 \\
    68 &   &   & 2 &  7 & 14 & 11 &  5 &  1 &    &    &   &   \\ 
    70 &   &   &   &  3 &  7 &  7 &  3 &  1 &    &    &   &   \\
    71 &   &   &   &    &  3 &  7 &  7 &  3 &  1 &    &   &   \\
    72 &   & 3 & 9 & 27 & 54 & 75 & 56 & 26 &  6 &  1 &   &   \\
    74 &   &   &   &  2 &  5 &  8 &  5 &  1 &    &    &   &   \\
    75 &   &   &   &    &  2 &  5 &  5 &  4 &  1 &    &   &   \\
    76 &   &   & 6 & 22 & 59 & 86 & 75 & 34 &  9 &  1 &   &   \\ 
    77 &   &   &   &    &  2 &  5 &  8 &  6 &  4 &  1 &   &   \\
    78 &   &   &   & 10 & 36 & 64 & 66 & 52 & 28 & 11 & 2 &   \\
    79 &   &   &   &    & 14 & 47 & 71 & 63 & 44 & 23 & 8 & 1 \\
    \hline      
    \end{tabular}
    \caption{Number of $16$-divisible codes per dimension $k$ and effective length $n$.}
  \end{center}
\end{table}

For ternary linear codes we can state $\Gamma_3(3,n)=1$ iff $n=4$ or $n\ge 8$. Moreover, we have 
$\Gamma_3(3,n)=3$ iff $n\in \{3,6,7\}$ and $\Gamma_3(3,n)=\infty$ iff $n\in \{1,2,5\}$. For 
quaternary linear codes we can state $\Gamma_4(4,n)=1$ iff $n\in\{5, 10, 15, 16, 17\}$ or $n\ge 20$. 
Moreover, we have $\Gamma_4(4,n)=2$ iff $n\in\{12,14,18,19\}$, $\Gamma_4(4,n)=4$ iff $n\in\{4,8,9,13\}$, and 
$\Gamma_4(4,n)=\infty$ iff $n\in\{1,2,3,6,7,11\}$. For $n\in\{6,7,9\}$ there exist projective $[n,3]_4$ two-weight 
codes that are $2$-divisible. They belong to the families $TF1$, $TF2$, $RT1$, and $RT2$, see \cite{calderbank1986geometry}. 
Since $19=14+5$ this gives constructions for all cases with $\Gamma_4(4,n)=2$.

\begin{table}[htp]
  \begin{center}
    \begin{tabular}{|r|rrr|}
    \hline
    n / k & 1 & 2 & 3 \\ 
    \hline
    3 & 1 &   &   \\ 
    4 &   & 1 &   \\
    6 & 1 & 1 &   \\ 
    7 &   & 1 & 1 \\
    \hline      
    \end{tabular}
    \caption{Number of $3$-divisible ternary linear codes per dimension $k$ and effective length $n$.}
  \end{center}
\end{table}

\begin{table}[htp]
  \begin{center}
    \begin{tabular}{|r|rrrrrrr|}
    \hline
    n / k & 1 & 2 & 3 & 4 & 5 & 6 & 7 \\ 
    \hline
     4 & 1 &   &    &    &    &   &   \\
     5 &   & 1 &    &    &    &   &   \\
     8 & 1 & 1 &    &    &    &   &   \\
     9 &   & 1 &  1 &    &    &   &   \\
    10 &   & 1 &  1 &  1 &    &   &   \\ 
    12 & 1 & 2 &  2 &    &    &   &   \\
    13 &   & 2 &  3 &  1 &    &   &   \\
    14 &   & 1 &  5 &  3 &  1 &   &   \\  
    15 &   & 1 &  3 &  6 &  2 & 1 &   \\ 
    16 & 1 & 4 &  9 &  7 &  2 &   &   \\
    17 &   & 3 & 12 &  9 &  2 &   &   \\ 
    18 &   & 2 & 18 & 25 &  8 & 1 &   \\ 
    19 &   & 1 & 14 & 42 & 25 & 6 & 1 \\
    \hline      
    \end{tabular}
    \caption{Number of $4$-divisible quaternary linear codes per dimension $k$ and effective length $n$.}
  \end{center}
\end{table}

\end{document}